\documentclass{amsart}
\usepackage[a4paper,hmargin={2.5cm,2.5cm},vmargin={2.5cm,2.5cm},heightrounded,marginparwidth=2.3cm,marginparsep=0.05cm]{geometry}

\usepackage[utf8]{inputenc}

\usepackage[square]{natbib} 
\usepackage[leqno]{amsmath}
\usepackage{amssymb,amsthm}
\usepackage[mathscr]{euscript}
\usepackage{bbm}
\usepackage{dsfont}
\usepackage{stmaryrd} 
\usepackage{enumitem}
\usepackage[all]{xy}
\usepackage{mathtools}

\makeatletter
\def\namedlabel#1#2{\begingroup
    #2%
    \def\@currentlabel{#2}%
    \phantomsection\label{#1}\endgroup
}
\makeatother

\theoremstyle{plain}
\newtheorem{theorem}{Theorem}[section]
\newtheorem{lemma}[theorem]{Lemma}
\newtheorem{proposition}[theorem]{Proposition}
\newtheorem{corollary}[theorem]{Corollary}

\theoremstyle{definition}
\newtheorem{definition}[theorem]{Definition}
\newtheorem{examples}[theorem]{Examples}
\newtheorem{example}[theorem]{Example}
\newtheorem{assumption}[theorem]{Assumption}

\theoremstyle{remark}
\newtheorem{remark}[theorem]{Remark}
\newtheorem{remarks}[theorem]{Remark}

\newlist{tfae}{enumerate}{1}
\setlist[tfae,1]{label=(\roman*)}

\usepackage{chngcntr}
\counterwithin*{equation}{section}

\newcommand{\fw}{\mathfrak{w}}
\newcommand{\fx}{\mathfrak{x}}
\newcommand{\fy}{\mathfrak{y}}

\newcommand{\calA}{\mathcal{A}}

\newcommand{\fX}{\mathfrak{X}}

\DeclareMathOperator{\upc}{\uparrow\!}
\DeclareMathOperator{\downc}{\downarrow\!}
\DeclareMathOperator{\ev}{ev}
\DeclareMathOperator{\supp}{supp}
\DeclareMathOperator{\Zero}{\mathcal{Z}}
\DeclareMathOperator{\Anti}{\mathcal{A}}
\DeclareMathOperator{\im}{im}

\newcommand{\ouc}[1]{\upc\overline{#1}}

\DeclareMathOperator{\Mnd}{Mon}
\DeclareMathOperator{\LaxMnd}{LaxMon}
\DeclareMathOperator{\CoAlg}{CoAlg}
\DeclareMathOperator{\Fix}{Fix}

\newcommand{\catfont}[1]{\mathsf{#1}}

\newcommand{\catX}{\catfont{X}}
\newcommand{\catA}{\catfont{A}}
\newcommand{\catB}{\catfont{B}}

\newcommand{\SET}{\catfont{Set}}
\newcommand{\REL}{\catfont{Rel}}
\newcommand{\ORD}{\catfont{Ord}}
\newcommand{\MET}{\catfont{Met}}
\newcommand{\UMET}{\catfont{UMet}}
\newcommand{\BMET}{\catfont{BMet}}
\newcommand{\TOP}{\catfont{Top}}
\newcommand{\ORDCH}{\catfont{OrdComp}}
\newcommand{\POSCH}{\catfont{PosComp}}
\newcommand{\POSCHDIST}{\catfont{PosCompDist}}
\newcommand{\STCOMP}{\catfont{StComp}}
\newcommand{\STONE}{\catfont{BooSp}}
\newcommand{\COMPHAUS}{\catfont{CompHaus}}
\newcommand{\COMPHAUSREL}{\catfont{CompHausRel}}
\newcommand{\APP}{\catfont{App}}
\newcommand{\SLAT}{\catfont{SLat}}
\newcommand{\BOOL}{\catfont{Boo}}
\newcommand{\BAO}{\catfont{BAO}}

\newcommand{\Rels}[1]{#1\text{-}\catfont{Rel}}
\newcommand{\Cats}[1]{#1\text{-}\catfont{Cat}}
\newcommand{\Reps}[1]{#1\text{-}\catfont{Rep}}

\newcommand{\Tens}[1]{#1\text{-}\catfont{CoPow}}
\newcommand{\FinCoCts}[1]{#1\text{-}\catfont{FinCoCts}}
\newcommand{\FinSup}[1]{#1\text{-}\catfont{FinSup}}
\newcommand{\POSCHDists}[1]{#1\text{-}\catfont{PosCompDist}}
\newcommand{\RepDists}[1]{#1\text{-}\catfont{RepDist}}

\newcommand{\two}{\catfont{2}}
\newcommand{\V}{\mathcal{V}}

\newcommand{\op}{\mathrm{op}}
\newcommand{\sep}{\mathrm{sep}}
\newcommand{\cc}{\mathrm{cc}}
\newcommand{\bool}{\mathrm{boo}}

\newcommand{\luk}{\odot}

\newcommand{\relto}{\mathrel{\mathmakebox[\widthof{$\xrightarrow{\rule{1.45ex}{0ex}}$}]
{\xrightarrow{\rule{1.45ex}{0ex}}\hspace*{-2.4ex}{\mapstochar}\hspace*{1.8ex}}}} 
\newcommand{\modto}{\mathrel{\mathmakebox[\widthof{$\xrightarrow{\rule{1.45ex}{0ex}}$}]
{\xrightarrow{\rule{1.45ex}{0ex}}\hspace*{-2.8ex}{\circ}\hspace*{1ex}}}} 

\newcommand\adjunctop[2]{\xymatrix@=8ex{\ar@{}[r]|{\bot}\ar@<1mm>@/^2mm/[r]^{{#2}} & \ar@<1mm>@/^2mm/[l]^{{#1}}}}
\newcommand\adjunct[2]{\xymatrix@=8ex{\ar@{}[r]|{\top}\ar@<1mm>@/^2mm/[r]^{{#2}} & \ar@<1mm>@/^2mm/[l]^{{#1}}}}

\newcommand{\monadfont}[1]{\mathbbm{#1}}

\newcommand{\mT}{\monadfont{T}}
\newcommand{\mD}{\monadfont{D}}

\newcommand{\mU}{\monadfont{U}}
\newcommand{\mP}{\monadfont{P}}
\newcommand{\mV}{\monadfont{V}}

\newcommand{\monad}{(T,m,e)}
\newcommand{\umonad}{(U,m,e)}
\newcommand{\vmonad}{(V,m,e)}

\newcommand{\theoryfont}[1]{\mathscr{#1}}
\newcommand{\thU}{\theoryfont{U}}

\newcommand{\utheory}{(\mU,[0,1],\xi)}

\newcommand{\Mon}{\ref{Mon}}
\newcommand{\Act}{\ref{Act}}
\newcommand{\Sup}{\ref{Sup}}
\newcommand{\LTen}{\ref{LTen}}
\newcommand{\Ten}{\ref{Ten}}
\newcommand{\Top}{\ref{Top}}
\newcommand{\Min}{\ref{Min}}

\newcommand{\CondA}{\ref{CondA}}

\newcommand{\doo}[1]{\overset{\centerdot}{#1}}

\newcommand{\df}[1]{\emph{\textbf{#1}}}

\newcommand{\field}[1]{\mathds{#1}}

\newcommand{\N}{\field{N}}

\newcommand{\R}{\field{R}}

\title{Enriched Stone-type dualities}

\author{Dirk Hofmann}

\author{Pedro Nora}

\thanks{Partial financial assistance by Portuguese funds through CIDMA
  (Center for Research and Development in Mathematics and
  Applications), and the Portuguese Foundation for Science and
  Technology (``FCT -- Funda\c{c}\~ao para a Ci\^encia e a
  Tecnologia''), within the project UID/MAT/04106/2013 is gratefully
  acknowledged. Pedro Nora is also supported by FCT grant
  SFRH/BD/95757/2013}

\address{Center for Research and Development in Mathematics and
  Applications, Department of Mathematics, University of Aveiro,
  3810-193 Aveiro, Portugal}

\email{dirk@ua.pt}
\email{a28224@ua.pt}

\date{\today}



\usepackage{hyperref} 

\begin{document}
\begin{abstract}
  A common feature of many duality results is that the involved
  equivalence functors are liftings of hom-functors into the
  two-element space resp.{} lattice. Due to this fact, we can only
  expect dualities for categories cogenerated by the two-element set
  with an appropriate structure. A prime example of such a situation
  is Stone's duality theorem for Boolean algebras and Boolean spaces,
  the latter being precisely those compact Hausdorff spaces which are
  cogenerated by the two-element discrete space. In this paper we aim
  for a systematic way of extending this duality theorem to categories
  including all compact Hausdorff spaces. To achieve this goal, we
  combine duality theory and quantale-enriched category theory. Our
  main idea is that, when passing from the two-element discrete space
  to a cogenerator of the category of compact Hausdorff spaces, all
  other involved structures should be substituted by corresponding
  enriched versions. Accordingly, we work with the unit interval
  $[0,1]$ and present duality theory for ordered and metric compact
  Hausdorff spaces and (suitably defined) finitely cocomplete
  categories enriched in $[0,1]$.
\end{abstract}

\maketitle

\tableofcontents

\section{Introduction}
\label{sec:intro}

In \citep{BD01}, the authors make the seemingly paradoxical
observation that ``\dots an equation is only interesting or useful to
the extent that the two sides are different!''. Undoubtedly, a
moment's thought convinces us that an equation like
$e^{i\omega}=\cos(\omega)+i\sin(\omega)$ is far more interesting than
the rather dull statement that $3=3$. A comparable remark applies if
we go up in dimension: equivalent categories are thought to be
essentially equal, but an equivalence is of more interest if the
involved categories look different. Numerous examples of equivalences
of ``different'' categories relate a category $\catX$ and the dual of
a category $\catA$. Such an equivalence is called a \df{dual
  equivalence} or simply a \df{duality}, and is usually denoted by
$\catX\simeq\catA^\op$.  As it is true for every equivalence, a
duality allows to transport properties from one side to the other. The
presence of the dual category on one side is often useful since our
knowledge of properties of a category is typically asymmetric. Indeed,
many ``everyday categories'' admit a representable and hence limit
preserving functor to $\SET$. Therefore in these categories limits are
``easy''; however, colimits are often ``hard''. In these
circumstances, an equivalence $\catX\simeq\catA^\op$ together with the
knowledge of limits in $\catA$ help us to understand colimits in
$\catX$. The dual situation, where colimits are ``easy'' and limits
are ``hard'', frequently emerges in the context of coalgebras. For
example, the category $\CoAlg(V)$ of coalgebras for the Vietoris
functor $V$ on the category $\STONE$ of Boolean spaces and continuous
functions is known to be equivalent to the dual of the category $\BAO$
with objects Boolean algebras $B$ with an operator $h:B\to B$
satisfying the equations
\begin{align*}
  h(\bot)=\bot &&\text{and}&& h(x\vee y)=h(x)\vee h(y),
\end{align*}
and morphisms the Boolean homomorhisms which also preserve the
additional unary operation (see \citep{Hal56}). It is a trivial
observation that $\BAO$ is a category of algebras over $\SET$ defined
by a (finite) set of operations and a collection of equations; every
such category is known to be complete and cocomplete. Notably, the
equivalence $\CoAlg(V)\simeq\BAO^\op$ allows to conclude the
non-trivial fact that $\CoAlg(V)$ is complete. This argument also
shows that, starting with a category $\catX$, the category $\catA$ in
a dual equivalence $\catX\simeq\catA^\op$ does not need to be a
familiar category. It is certainly beneficial that $\catA=\BAO$ is a
well-known category; however, every algebraic category describable by
a set of operations would be sufficient to conclude completeness of
$\catX=\CoAlg(V)$. We refer to \citep{KKV04,BKR07} for more examples of
dualities involving categories of coalgebras.

The example above as well as the classical Stone-dualities for Boolean
algebras and distributive lattices (see \citep{Sto36,Sto38a,Sto38})
are obtained using the two-element space or the two-element
lattice. Due to this fact, we can only expect dualities for categories
cogenerated by $\two=\{0,1\}$ with an appropriate structure. For
instance, the category $\STONE$ is the full subcategory of the
category $\COMPHAUS$ of compact Hausdorff spaces and continuous maps
defined by those spaces $X$ where the cone $(f:X\to\two)_f$ is
point-separating and initial. In order to obtain duality results
involving all compact Hausdorff spaces, we need to work with a
cogenerator of $\COMPHAUS$ rather than the 2-element discrete
space. Of course, this is exactly the perspective taken in the
classical Gelfand duality theorem (see \citep{Gel41a}) or in several
papers on lattices of continuous functions (see \citep{Kap47,Kap48}
and \citep{Ban83}) that consider functions into the unit disc or the
unit interval. However, in these approaches, the objects of the dual
category of $\COMPHAUS$ do not appear immediately as generalisations
of Boolean algebras.

This is the right moment to mention another cornerstone of our work:
the theory of quantale-enriched categories. Our main motivation stems
from Lawvere's seminal paper \citep{Law73} that investigates metric
spaces as categories enriched in the quantale $[0,\infty]$. Keeping in
mind that ordered sets are categories enriched in the two-element
quantale $\two$, our thesis is \emph{that the passage from the
  two-element space to the compact Hausdorff space $[0,\infty]$ should
  be matched by a move from ordered structures to metric structures on
  the other side}. In fact, we claim that some results about lattices
of real-valued continuous functions secretly talk about (ultra)metric
spaces; for instance, in Section~\ref{sec:enriched-cats-act}, we point
out how to interpret Propositions 2 and 3 of \citep{Ban83} in this
way. Roughly speaking, in analogy with the results for the two-element
space, we are looking for an equivalence functor (or at least a full
embedding)
\[
  \COMPHAUS\longrightarrow(\text{metric spaces with some
    (co)completeness properties})^\op
\]
and, more generally, with $\POSCH$ denoting the category of partially
ordered compact (Hausdorff) spaces and monotone continuous maps, a
full embedding
\[
  \POSCH\longrightarrow(\text{metric spaces with some
    (co)completeness properties})^\op.
\]
Inspired by \citep{Hal56}, we obtain this as a restriction of a more
general result about a full embedding of the Kleisli category
$\POSCH_{\mV}$ of the Vietoris monad $\mV$ on $\POSCH$:
\[
  \POSCH_{\mV}\longrightarrow(\text{``finitely cocomplete'' metric
    spaces})^\op.
\]
The notion of ``finitely cocomplete metric space'' should be
considered as the metric counterpart to semi-lattice, and ``metric
space with some (co)completeness properties'' as the metric
counterpart to (distributive) lattice. This way we also exhibit the
algebraic nature of the dual category of $\POSCH$ which resembles the
classical result stating that $\COMPHAUS^\op$ is a $\aleph_1$-ary
variety (see \citep{Isb82,MR17}).

For technical reasons, we consider structures enriched in a quantale
based on $[0,1]$ rather than in $[0,\infty]$; nevertheless, since the
lattices $[0,1]$ and $[0,\infty]$ are isomorphic, we still talk about
metric spaces. In Section~\ref{sec:enriched-cats-act} we recall the
principal facts about quantale-enriched categories needed in this
paper, and in Section~\ref{sec:cont-quantale-unit} we present the
classification of continuous quantale structures on the unit interval
$[0,1]$ obtained in \citep{Fau55} and \citep{MS57}. Since the Vietoris
monad $\mV$ on the category $\POSCH$ of partially ordered compact
spaces and monotone continuous maps plays a key role in the results of
Section~\ref{sec:duality-distributors}, we provide the necessary
background material in Section~\ref{sec:stably-compact-Vietoris}. We
review duality theory in Section~\ref{sec:dual-adjunctions}; in
particular, for a monad $\mT$, we explain the connection between
functors $\catX_\mT\to\catA^\op$ and certain $\mT$-algebras. After
these introductory part, in Section~\ref{sec:duality-distributors} we
develop duality theory for the Kleisli category $\POSCH_\mV$. We found
a first valuable hint in \citep{Sha92} where the author gives a
functional representation of the classical Vietoris monad on
$\COMPHAUS$ using the algebraic structure on the non-negative
reals. Inspired by this result, for every continuous quantale
structure on $[0,1]$, we obtain a functional representation of the
Vietoris monad on $\POSCH$, which leads to a full embedding of
$\POSCH_\mV$ into a category of monoids of finitely cocomplete
$[0,1]$-categories. We also identify the continuous functions in
$\POSCH_\mV$ as precisely the monoid homomorphisms on the other
side. Section~\ref{sec:stone-weierstrass-theorem} presents a
Stone--Weierstra\ss{} type theorem for $[0,1]$-categories which helps
us to establish a dual equivalence involving the category
$\POSCH_\mV$. Finally, since we moved from order structures to
structures enriched in $[0,1]$, it is only logical to also substitute
the Vietoris monad, which is based on functions $X\to\two$, by a monad
that uses functions of type $X\to [0,1]$ defined on metric
generalisations of partially ordered compact spaces. In
Sections~\ref{sec:metr-comp-hausd} and \ref{sec:enriched-Vietoris} we
extend our setting from partially ordered compact spaces to ``metric
compact Hausdoff spaces'' and consider the enriched Vietoris monads
introduced in \citep{Hof14}. Denoting these monads by $\mV$ as well,
in analogy to the ordered case, for certain quantale structures on
$[0,1]$ we obtain a full embedding
\[
  (\text{metric compact Hausdoff
    spaces})_{\mV}\longrightarrow(\text{``finitely cocomplete''
    $[0,1]$-categories})^\op.
\]

Last but not least, we would like to point out that this is not the
first work transporting classical duality results into the realm of
metric spaces. An approach version (see \citep{Low97}) of the duality
between the categories of sober spaces and continuous maps and of
spatial frames and homomorphisms is obtained in \citep{BLO06} and
extensively studied in \citep{Olm05} (see also \citep{OV10}). By
definition, an approach frame is a frame with some actions of
$[0,\infty]$; keeping in mind the results of
Section~\ref{sec:enriched-cats-act}, we can describe approach frames
as particular (co)complete metric spaces. This point of view is taken
in \citep{HS11}.

\section{Enriched categories as actions}\label{sec:enriched-cats-act}

To explain the passage from the ordered to the metric case, it is
convenient to view ordered sets and metric spaces as instances of the
same notion, namely that of a quantale-enriched category. All material
presented here is well-known, we refer to the classical sources
\citep{EK66}, \citep{Law73} and \citep{Kel82}. A very extensive
presentation of this theory in the quantaloid-enriched case can be
found in \citep{Stu05,Stu06,Stu07a}. We would also like to point the
reader to \citep{KL00}, \citep{KS05} and \citep{CH09a} where enriched
categories with certain colimits are studied.

\begin{definition}
  A (commutative and unital) \df{quantale} $\V$ is a complete lattice
  which carries the structure of a commutative monoid
  $\otimes:\V\times\V\to\V$ with unit element $k\in\V$ such that
  $u\otimes-:\V\to\V$ preserves suprema, for each $u\in\V$.
\end{definition}

Hence, every monotone map $u\otimes-:\V\to\V$ has a right adjoint
$\hom(u,-):\V\to\V$ which is characterised by
\[
  u\otimes v\le w\iff v\le\hom(u,w),
\]
for all $v,w\in\V$.

\begin{definition}
  Let $\V$ be a quantale. A \df{$\V$-category} is a pair $(X,a)$
  consisting of a set $X$ and a map $a:X\times X\to\V$ satisfying
  \begin{align*}
    k&\le a(x,x), & a(x,y)\otimes a(y,z)\le a(x,z),
  \end{align*}
  for all $x,y,z\in X$. Given $\V$-categories $(X,a)$ and $(Y,b)$, a
  \df{$\V$-functor} $f:(X,a)\to (Y,b)$ is a map $f:X\to Y$ such that
  \[
    a(x,y)\le b(f(x),f(y)),
  \]
  for all $x,y\in X$.
\end{definition}

For every $\V$-category $(X,a)$, $a^\circ(x,y)=a(y,x)$ defines another
$\V$-category structure on $X$, and the $\V$-category
$(X,a)^\op:=(X,a^\circ)$ is called the \df{dual} of $(X,a)$. Clearly,
$\V$-categories and $\V$-functors define a category, denoted as
$\Cats{\V}$. The category $\Cats{\V}$ is complete and cocomplete, and
the canonical forgetful functor $\Cats{\V}\to\SET$ preserves limits
and colimits. The quantale $\V$ becomes a $\V$-category with structure
$\hom:\V\times\V\to\V$. For every set $S$, we can form the $S$-power
$\V^S$ of $\V$ which has as underlying set all functions $h:S\to\V$,
and the $\V$-category structure $[-,-]$ is given by
\[
  [h,l]=\bigwedge_{s\in S}\hom(h(s),l(s)),
\]
for all $h,k:S\to\V$.

For a quantale $\V$ and sets $X$, $Y$, a \df{$\V$-relation} from $X$
to $Y$ is a map $X \times Y \to \V$ and it will be represented by
$X \relto Y$. As for ordinary relations, $\V$-relations can be
composed via "matrix mulptiplication". That is, for $r: X \relto Y$
and $s: Y \relto Z$, the composite $s \cdot r:X \relto Z$ is
calculated pointwise by
\[
  (s \cdot r)(x,z)=\bigvee_{y \in Y} r(x,y) \otimes s(y,z),
\]
for every $x \in X$ and $z \in Z$. We note that the structure of a
$\V$-category is by definition a reflexive and transitive
$\V$-relation, since the axioms dictate that, for a $\V$-category
$(X,a)$, $1_X \leq a$ and $a \cdot a \leq a$. A $\V$-relation
$r:X \relto Y$ between $\V$-categories $(X,a)$ and $(Y,b)$ is called a
\df{$\V$-distributor} if $ r \cdot a \leq r $ and $ b \cdot r \leq r$,
and we write $r:(X,a) \modto (Y,b)$. Clearly, the reverse inequalities
always hold, which means that a $\V$-distributor is a $\V$-relation
that is strictly compatible with the structure of $\V$-category.

\begin{examples}\label{exs:quantales}
  Our principal examples are the following.
  \begin{enumerate}
  \item The two-element Boolean algebra $\two=\{0,1\}$ of truth values
    with $\otimes$ given by ``and'' $\&$. Then
    $\hom(u,v)=(u\implies v)$ is implication. The category
    $\Cats{\two}$ is equivalent to the category $\ORD$ of ordered sets
    and monotone maps.
  \item The complete lattice $[0,\infty]$ ordered by the ``greater or
    equal'' relation $\geqslant$ (so that the infimum of two numbers
    is their maximum and the supremum of $S\subseteq[0,\infty]$ is
    given by $\inf S$), with multiplication $\otimes=+$. In this case
    we have
    \[\hom(u,v)=v\ominus u:=\max(v-u,0).\]
    For this quantale, a $[0,\infty]$-category is a generalised metric
    space \`a la \citeauthor{Law73} and a $[0,\infty]$-functor is a
    non-expansive map (see \citep{Law73}). We denote this category by
    $\MET$.
  \item\label{item:4} Of particular interest to us is the complete
    lattice $[0,1]$ with the usual ``less or equal'' relation $\le$,
    which is isomorphic to $[0,\infty]$ via the map
    $[0,1]\to[0,\infty],\,u\mapsto -\ln(u)$ where $-\ln(0)=\infty$. As
    the examples below show, metric, ultrametric and bounded metric
    spaces appear as categories enriched in a quantale based on this
    lattice. More in detail, we consider the following quantale
    operations on $[0,1]$ with neutral element $1$.
    \begin{enumerate}
    \item The tensor $\otimes=*$ is multiplication and then
      \[
        \hom(u,v)=v\varoslash u:=
        \begin{cases}
          \min(\frac{v}{u},1) & \text{if $u\neq 0$,}\\
          1 & \text{otherwise.}
        \end{cases}
      \]
      Via the isomorphism $[0,1]\simeq[0,\infty]$, this quantale is
      isomorphic to the quantale $[0,\infty]$ described above, hence
      $\Cats{[0,1]}\simeq\MET$.
    \item The tensor $\otimes=\wedge$ is infimum and then
      \[
        \hom(u,v)=
        \begin{cases}
          1 & \text{if $u\le v$,}\\
          v & \text{otherwise.}
        \end{cases}
      \]
      In this case, the isomorphism $[0,1]\simeq[0,\infty]$
      establishes an equivalence between $\Cats{[0,1]}$ and the
      category $\UMET$ of ultrametric spaces and non-expansive maps.
    \item The tensor $\otimes=\luk$ is the \df{\L{}ukasiewicz tensor}
      given by $u\luk v=\max(0,u+v-1)$, here
      $\hom(u,v)=\min(1,1-u+v)=1-\max(0,u-v)$. Via the lattice
      isomorphism $[0,1]\to[0,1],\,u\mapsto 1-u$, this quantale is
      isomorphic to the quantale $[0,1]$ with ``greater or equal''
      relation $\geqslant$ and tensor $u\otimes v=\min(1,u+v)$
      truncated addition. This observation identifies $\Cats{[0,1]}$
      as the category $\BMET$ of bounded-by-$1$ metric spaces and
      non-expansive maps.
    \end{enumerate}
  \end{enumerate}
\end{examples}

Every $\V$-category $(X,a)$ carries a natural order defined by
\[
  x\le y \text{ whenever } k\le a(x,y),
\]
which can be extended pointwise to $\V$-functors making $\Cats{\V}$ a
\emph{2-category}. Therefore we can talk about adjoint $\V$-functors;
as usual, $f:(X,a)\to(Y,b)$ is left adjoint to $g:(Y,b)\to(X,a)$,
written as $f\dashv g$, whenever $1_X\le gf$ and $fg\le
1_Y$. Equivalently, $f\dashv g$ if and only if
\[
  b(f(x),y)=a(x,g(y)),
\]
for all $x\in X$ and $y\in Y$. We note that maps $f$ and $g$ between
$\V$-categories satisfying the equation above are automatically
$\V$-functors.

The natural order of $\V$-categories defines a faithful functor
$\Cats{\V}\to\ORD$. A $\V$-category is called \df{separated} whenever
its underlying ordered set is anti-symmetric, and we denote by
$\Cats{\V}_\sep$ the full subcategory of $\Cats{\V}$ defined by all
separated $\V$-categories. In particular, $\ORD_\sep$ denotes the
category of all anti-symmetric ordered sets and monotone maps. In the
sequel we will frequently consider separated $\V$-categories in order
to guarantee that adjoints are unique. We note that the underlying
order of the $\V$-category $\V$ is just the order of the quantale
$\V$, and the order of $\V^S$ is calculated pointwise. In particular,
$\V^S$ is separated.

\begin{definition}
  A $\V$-category $(X,a)$ is called \df{$\V$-copowered} whenever the
  $\V$-functor $a(x,-):(X,a)\to(\V,\hom)$ has a left adjoint
  $x\otimes-:(\V,\hom)\to(X,a)$ in $\Cats{\V}$, for every $x\in X$.
\end{definition}

Elementwise, this means that, for all $x\in X$ and $u\in\V$, there
exists some element $x\otimes u\in X$, called the $u$-copower of $x$,
such that
\[
  a(x\otimes u,y)=\hom(u,a(x,y)),
\]
for all $y\in X$.

\begin{example}
  The $\V$-category $\V$ is $\V$-copowered, with copowers given by the
  multiplication of the quantale $\V$. More generally, for every set
  $S$, the $\V$-category $\V^S$ is $\V$-copowered: for every
  $h\in\V^S$ and $u\in\V$, the $u$-copower of $h$ is given by
  \[
    (h\otimes u)(x)=h(x)\otimes u,
  \]
  for all $x\in S$.
\end{example}

\begin{remark}
  If $(X,a)$ is a $\V$-copowered $\V$-category, then, for every
  $x\in X$ and $u=\bot$ the bottom element of $\V$, we have
  \[
    a(x\otimes \bot,y)=\hom(\bot,a(x,y))=\top
  \]
  for all $y\in X$. In particular, $x\otimes \bot$ is a bottom element
  of the $\V$-category $(X,a)$.
\end{remark}

Every $\V$-copowered and separated $\V$-category comes equipped with
an action $\otimes:X\times\V\to X$ of the quantale $\V$ satisfying
\begin{enumerate}
\item $x\otimes k=x$,
\item $(x\otimes u)\otimes v=x\otimes(u\otimes v)$,
\item
  $\displaystyle{x\otimes\bigvee_{i\in I}u_i=\bigvee_{i\in I}(x\otimes
    u_i)}$;
\end{enumerate}
for all $x\in X$ and $u,v,u_i\in \V$ ($i\in I$). Conversely, given an
anti-symmetric ordered set $X$ with an action $\otimes:X\times\V\to X$
satisfying the three conditions above, one defines a map
$a:X\times X\to\V$ by $x\otimes-\dashv a(x,-)$, for all $x\in X$. It
is easy to see that $(X,a)$ is a $\V$-copowered $\V$-category whose
order is the order of $X$ and where copowers are given by the action
of $X$. Writing $\Tens{\V}_{\sep}$ for the category of $\V$-copowered
and separated $\V$-categories and copower-preserving $\V$-functors and
$\ORD_{\sep}^\V$ for the category of anti-symmetric ordered sets $X$
with an action $\otimes:X\times\V\to X$ satisfying the three
conditions above and action-preserving monotone maps, the above
construction yields an isomorphism
\[
  \Tens{\V}_{\sep}\simeq\ORD_{\sep}^\V.
\]
We also note that the inclusion functor $\Tens{\V}_{\sep}\to\Cats{\V}$
is monadic.

\begin{remark}\label{rem:Ban}
  The identification of certain metric spaces as ordered sets with an
  action of $[0,1]$ allows us to spot the appearance of metric spaces
  where it does not seem obvious at first sight. For instance,
  \citep{Ban83} considers the distributive lattice $DX$ of continuous
  functions from a compact Hausdorff space $X$ into the unit interval
  $[0,1]$, and, for a continuous map $f:X\to Y$, the lattice
  homomorphism $Df:DY\to DX,\,\psi\mapsto\psi\cdot f$ is given by
  composition with $f$. In \citep[Proposition 2]{Ban83} it is shown
  that a lattice homomorphism $\varphi:DY\to DX$ is of the form
  $\varphi=Df$, for some continuous map $f:X\to Y$, if and only if
  $\varphi$ preserves constant functions. Subsequently, \citep{Ban83}
  consideres the algebraic theory of distributive lattices augmented
  by constants, one for each element of $[0,1]$; and eventually
  obtains a duality result for compact Hausdorff spaces. Motivated by
  the considerations in this section, instead of adding constants we
  will consider $DX$ as a lattice equipped with the action of $[0,1]$
  defined by
  \[
    (f\otimes u)(x)=f(x)\wedge u,
  \]
  and then \citep[Proposition 2]{Ban83} tells us that the lattice
  homomorphisms $\varphi:DY\to DX$ of the form $\varphi=Df$ are
  precisely the action-preserving ones. Hence, \citeauthor{Ban83}'s
  result can be reinterpreted in terms of $[0,1]$-copowered
  ultrametric spaces.
\end{remark}

The notion of copower in a $\V$-category $(X,a)$ is a special case of
a weighted colimit in $(X,a)$, as we recall next. In the remainder of
this section we write $G$ to denote the $\V$-category $(1,k)$.

A \df{weighted colimit diagram} in $X$ is given by a $\V$-category $A$
together with a $\V$-functor $h:A\to X$ and a $\V$-distributor
$\psi:A\modto G$, the latter is called the \df{weight} of the
diagram. A \df{colimit} of such a diagram is an element $x_0\in X$
such that, for all $x\in X$,
\[
  a(x_0,x)=\bigwedge_{z\in A}\hom(\psi(z),a(h(z),x)).
\]
If a weighted colimit diagram has a colimit, then this colimit is
unique up to equivalence. A $\V$-functor $f:X\to Y$ between
$\V$-categories \df{preserves} the colimit of this diagram whenever
$f(x_0)$ is a colimit of the weighted colimit diagram in $Y$ given by
$fh:A\to Y$ and $\psi:A\modto G$.
\begin{examples}
  \begin{enumerate}
  \item For $A=G$, a weighted colimit diagram in $X$ is given by an
    element $x:G\to X$ and an element $u:G\modto G$ in $\V$, a colimit
    of this diagram is the $u$-copower $x\otimes u$ of $x$.
  \item For a family $h:I\to X,\,i\mapsto x_i$ in $X$ we consider the
    distributor $\psi:I\modto G$ defined by $\psi(z)=k$, for all
    $z\in I$. Then $\bar{x}$ is a colimit of this diagram precisely
    when
    \[
      a(\bar{x},x)=\bigwedge_{i\in I}a(x_i,x),
    \]
    for all $x\in X$; that is, $\bar{x}$ is an order-theoretic
    supremum of $(x_i)_{i\in I}$ and every $a(-,x):X\to\V^\op$
    preserves this supremum. Such a supremum is called \df{conical
      supremum}.

    Recall that a $\V$-copowered $\V$-category $(X,a)$ can be viewed
    as an ordered set $X$ with an action $\otimes:X\times\V\to X$. In
    terms of this structure, $(X,a)$ has all conical suprema of a
    given shape $I$ if and only if every family $(x_i)_{i\in I}$ has a
    supremum in the ordered set $X$ and, moreover,
    \[
      \left(\bigvee_{i\in I}x_i\right)\otimes u\simeq\bigvee_{i\in
        I}(x_i\otimes u)
    \]
    for all $u\in\V$. This follows from the facts that
    $\bigvee_{i\in I}x_i\otimes-$ is left adjoint to
    $a(\bigvee_{i\in I}x_i,-)$ and
    \[
      \V\xrightarrow{\,\Delta_\V\,}\V^I\xrightarrow{\,\prod_{i\in
          I}(x_i\otimes-)\,}X^I\xrightarrow{\,\bigvee\,}X
    \]
    is left adjoint to
    \[
      X\xrightarrow{\,\Delta_X\,}X^I\xrightarrow{\,\prod_{i\in
          I}(a(x_i,-)\,}\V^I\xrightarrow{\,\bigwedge\,}\V.
    \]
  \end{enumerate}
\end{examples}
A $\V$-category $X$ is called \df{cocomplete} if every weighted
colimit diagram has a colimit in $X$. One can show that $X$ is
cocomplete if and only if $X$ has the two types of colimits described
above, in this case the colimit of an arbitrary diagram can be
calculated as
\[
  x_0=\bigvee_{z\in A}h(z)\otimes\psi(z)
\]
since
\[
  a(\bigvee_{z\in A}h(z)\otimes\psi(z),x)=\bigwedge_{z\in
    A}a(h(z)\otimes\psi(z),x)=\bigwedge_{z\in
    A}\hom(\psi(z),a(h(z),x)).
\]
In particular, the $\V$-category $\V$ is cocomplete, and so are all
its powers $\V^S$.

A $\V$-functor $f:X\to Y$ between cocomplete $\V$-categories is called
\df{cocontinuous} whenever $f$ preserves all colimits of weighted
colimit diagrams; by the above, $f$ is cocontinuous if and only if $f$
preserves copowers and order-theoretic suprema.

\begin{definition}
  A $\V$-category $X$ is called \df{finitely cocomplete} whenever
  every weighted colimit diagram given by $h:A\to X$ and
  $\psi:A\modto G$ where the underlying set of $A$ is finite has a
  colimit in $X$. We call a $\V$-functor $f:X\to Y$ between finitely
  cocomplete $\V$-categories \df{finitely cocontinuous} whenever those
  colimits are preserved.
\end{definition}

Therefore:
\begin{itemize}
\item $X$ is finitely cocomplete if and only if $X$ has all copowers,
  a bottom element, all order-theoretic binary suprema and, moreover,
  these suprema are preserved by all $\V$-functors
  $a(-,x):X\to\V^\op$.
\item a map $f:X\to Y$ between finitely cocomplete $\V$-categories is
  a finitely cocontinuous $\V$-functor if and only if $f$ is monotone
  and preserves copowers and binary suprema.
\end{itemize}

In the sequel we write $\FinSup{\V}$ to denote the category of
separated finitely cocomplete $\V$-categories and finitely
cocontinuous $\V$-functors. We also recall that the inclusion functor
$\FinSup{\V}\to\Cats{\V}$ is monadic; in particular, $\FinSup{\V}$ is
complete and $\FinSup{\V}\to\Cats{\V}$ preserves all limits.

\begin{remark}\label{rem:FinCoCts_quasivariety}
  By the considerations of this section, $\FinSup{\V}$ can be also
  seen as a quasivariety (for more information on algebraic categories
  we refer to \citep{AR94} and \citep{ARV10}). In fact, a separated
  finitely cocomplete $\V$-category can be described as a set $X$
  equipped with a nullary operation $\bot$, a binary operation $\vee$,
  and unary operations $-\otimes u$ ($u\in\V$), subject to the following
  equations and implications:
  \begin{align*}
    x\vee x&=x, &x\vee y&=y\vee x, & x\vee (y\vee z) &=(x\vee y)\vee z, & x\vee\bot&=x,\\
    x\otimes k&=x, & (x\otimes u)\otimes v&=x\otimes (u\otimes v),& \bot\otimes u&=\bot, & (x\vee y)\otimes u &=(x\otimes u)\vee(y\otimes u);
  \end{align*}
  for all $x,y,z\in X$ and $u,v\in\V$. We also have to impose the
  conditions
  \[
    x\otimes v=\bigvee_{u\in S}(x\otimes u),
  \]
  for all $x\in X$ and $S\subseteq\V$ with $v=\bigvee S$; however,
  this is not formulated using just the operations above. Writing
  $x\le y$ as an abbreviation for the equation $y=x\vee y$, we can
  express the condition ``$x\otimes v$ is the supremum of
  $\{x\otimes u\mid u\in S\}$'' by the equations
  \[
    x\otimes u\le x\otimes v,\hspace{2em}(u\in S)
  \]
  and the implication
  \[
    \bigwedge_{u\in S}(x\otimes u\le y)\Longrightarrow (x\otimes v\le
    y).
  \]
  Furthermore, the morphisms of $\FinSup{\V}$ correspond precisely to
  the maps preserving these operations. By the considerations above,
  with $\lambda$ denoting the smallest regular cardinal larger than
  the size of $\V$, the category $\FinSup{\V}$ is equivalent to a
  $\lambda$-ary quasivariety. From that we conclude that $\FinSup{\V}$
  is also cocomplete. Finally, if the quantale $\V$ is based on the
  lattice $[0,1]$, then it is enough to consider countable subsets
  $S\subseteq\V$, and therefore $\FinSup{\V}$ is equivalent to a
  $\aleph_1$-ary quasivariety.
\end{remark}

Another important class of colimit weights is the class of all right
adjoint $\V$-distributors $\psi:A\modto G$.

\begin{definition}
  A $\V$-category $X$ is called \df{Cauchy-complete} whenever every
  diagram $(h:A\to X,\psi:A\modto G$ with $\psi$ right adjoint has a
  colimit in $X$.
\end{definition}

The designation ``Cauchy-complete'' has its roots in
\citeauthor{Law73}'s amazing observation that, for metric spaces
interpreted as $[0,\infty]$-categories, this notion coincides with the
classical notion of Cauchy-completeness (see \citep{Law73}). We hasten
to remark that every $\V$-functor preserves colimits weighted by right
adjoint $\V$-distributors.

In this context, \citep{HT10} introduces a closure operator
$\overline{(-)}$ on $\Cats{\V}$ which facilitates working with
Cauchy-complete $\V$-categories. As usual, a subset $M\subseteq X$ of
a $\V$-category $(X,a)$ is \df{closed} whenever $M=\overline{M}$ and
$M$ is \df{dense} in $X$ whenever $\overline{M}=X$. Below we recall
the relevant facts about this closure operator.
\begin{theorem}\label{thm:closure}
  The following assertions hold.
  \begin{enumerate}
  \item For every $\V$-category $(X,a)$, $x\in X$ and $M\subseteq X$,
    \[
      x\in\overline{M} \iff k\le\bigvee_{z\in M}a(x,z)\otimes a(z,x).
    \]
    If $\V$ is completely distributive (see \citep{Ran52} and
    \citep{Woo04}) with totally below relation $\ll$ and
    $k\le\bigvee_{u\ll k}u\otimes u$, then $x\in\overline{M}$ if and
    only if, for every $u\ll k$, there is some $z\in M$ with
    $u\le a(x,z)$ and $u\le a(z,x)$. By \citep[Theorem 1.12]{Fla92},
    the quantale $\V$ satisfies $k\le\bigvee_{u\ll k}u\otimes u$
    provided that the subset $A=\{u\in\V\mid u\ll k\}$ of $\V$ is
    directed.
  \item The $\V$-category $\V$ is Cauchy-complete.
  \item The full subcategory of $\Cats{\V}$ defined by all
    Cauchy-complete $\V$-categories is closed under limits in
    $\Cats{\V}$.
  \item Let $X$ be a Cauchy-complete and separated $\V$-category and
    $M\subseteq X$. Then the $\V$-subcategory $M$ of $X$ is
    Cauchy-complete if and only if the subset $M\subseteq X$ is closed
    in $X$.
  \end{enumerate}
\end{theorem}

The notion of weighted colimit is dual to the one of weighted limit,
of the latter we only need the special case of $u$-powers, with
$u\in\V$.

\begin{definition}
  A $\V$-category $(X,a)$ is called \df{$\V$-powered} whenever, for
  every $x\in X$, the $\V$-functor $a(-,x):(X,a)^\op\to(\V,\hom)$ has
  a left adjoint in $\Cats{\V}$.
\end{definition}

Elementwise, this amounts to saying that, for every $x\in X$ and every
$u\in\V$, there is an element $x\pitchfork u\in X$, called the
\df{$u$-power} of $x$, satisfying
\[
  \hom(u,a(y,x))=a(y,x\pitchfork u),
\]
for all $y\in X$. The $\V$-category $\V$ is $\V$-powered where
$w\pitchfork u=\hom(u,w)$, more generally, $\V^S$ is $\V$-powered with
\[
  (h\pitchfork u)(x)=\hom(u,h(x)),
\]
for all $h\in\V^S$, $u\in\V$ and $x\in S$.

\begin{remark}\label{d:rem:1}
  For every $\V$-functor $f:X\to Y$, $x\in X$ and $u\in\V$,
  $f(u\pitchfork x)\le u\pitchfork f(x)$.
\end{remark}

\section{Continuous quantale structures on the unit interval}\label{sec:cont-quantale-unit}

In this paper we are particularly interested in quantales based on the
complete lattice $[0,1]$. We succinctly review the classification of
all \emph{continuous} quantale structures
$\otimes:[0,1]\times[0,1]\to[0,1]$ on $[0,1]$ with neutral element
$1$. Such quantale structures are also called \df{continuous
  t-norms}. The results obtained in \citep{Fau55} and \citep{MS57}
show that every such tensor is a combination of the three structures
mentioned in Examples~\ref{exs:quantales}(\ref{item:4}). A more
detailed presentation of this material is in \citep{AFS06}.

We start by recalling some standard notation. An element $x\in [0,1]$
is called \df{idempotent} whenever $x\otimes x=x$ and \df{nilpotent}
whenever $x\neq 0$ and, for some $n\in \N$, $x^n=0$. The number of
idempotent and nilpotent elements characterises the three tensors
$\wedge$, $\luk$ and $\otimes$ on $[0,1]$ among all continuous
t-norms.

\begin{proposition}
  Assume that $0$ and $1$ are the only idempotent elements of $[0,1]$
  with respect to a given continuous t-norm. Then
  \begin{enumerate}
  \item $[0,1]$ has no nilpotent elements, then $\otimes=*$ is
    multiplication.
  \item $[0,1]$ has a nilpotent element, then $\otimes=\luk$ is the
    \L{}ukasiewicz tensor. In this case, every element $x$ with
    $0<x<1$ is nilpotent.
  \end{enumerate}
\end{proposition}

To deal with the general case, for a continuous t-norm $\otimes$
consider the subset $E=\{x\in [0,1]\mid x\text{ is
  idempotent}\}$. Note that $E$ is closed in $[0,1]$ since it can be
presented as an equaliser of the diagram
\[
  \xymatrix{[0,1]\ar@<2pt>[r]^{\text{\tiny identity}}\ar@<-2pt>[r]_{-\otimes-} & [0,1]}
\]
in $\COMPHAUS$.

\begin{lemma}
  Let $\otimes$ be a continuous t-norm on $[0,1]$, $x,y\in[0,1]$ and
  $e\in E$ so that $x\le e\le y$. Then $x\otimes y=x$.
\end{lemma}

\begin{corollary}
  Let $\otimes$ be a continuous t-norm on $[0,1]$ so that every
  element is idempotent. Then $\otimes=\wedge$. 
\end{corollary}

Before announcing the main result of this section, we note that, for
idempotents $e<f$ in $[0,1]$, the closed interval $[e,f]$ is a
quantale with tensor defined by the restriction of the tensor on
$[0,1]$ and neutral element $f$.

\begin{theorem}\label{thm.cont_quantale_unit}
  Let $\otimes$ be a continuous t-norm on $[0,1]$. For every
  non-idempotent $x\in[0,1]$, there exist idempotent elements
  $e,f \in [0,1]$, with $e< x< f$, such that the quantale $[e,f]$ is
  isomorphic to the quantale $[0,1]$ with either multiplication or
  \L{}ukasiewicz tensor.
\end{theorem}

\begin{remark}
  We note that every isomorphism $[e,f]\to[0,1]$ of quantales is
  necessarily a homeomorphism.
\end{remark}

The following consequence of Theorem~\ref{thm.cont_quantale_unit} will
be useful in the sequel.

\begin{corollary}\label{cor:no-zero-div}
  Let $(u,v)\in[0,1]\times[0,1]$ with $u\otimes v=0$. Then either
  $u=0$ or $v^n=0$, for some $n\in\N$. Hence, if there are no
  nilpotent elements, then $u=0$ or $v=0$.
\end{corollary}
\begin{proof}
  Assume $u>0$. The assertion is clear if there is some idempotent $e$
  with $0<e\le u$. If there is no $e\in E$ with $0<e\le u$, then there
  is some $f\in E$ with $u<f$ and $[0,f]$ is isomorphic to $[0,1]$
  with either multiplication or \L{}ukasiewicz tensor. Since
  $u\otimes v=0$, $v<f$. If $[0,f]$ is isomorphic to $[0,1]$ with
  multiplication, then $v=0$; otherwise there is some $n\in\N$ with
  $v^n=0$.
\end{proof}

Conversely, continuous quantale structures on $[0,1]$ can be defined
piecewise:

\begin{theorem}
  Let $\otimes_i$ ($i\in I$) be a family of continuous quantale
  structures on $[0,1]$ with neutral element $1$ and $(a_i)_{i\in I}$
  and $(b_i)_{i\in I}$ be families of elements of $[0,1]$ so that the
  open intervals $]a_i,b_i[$ are pairwise disjoint. The operation
  \[
    x\otimes y=
    \begin{cases}
      a_i+(b_i-a_i)\cdot\left(\left(\frac{x-a_i}{b_i-a_i}\right)\otimes_i\left(\frac{y-a_i}{b_i-a_i}\right)\right) & \text{if $x,y\in[a_i,b_i]$,}\\
      x\wedge y & \text{otherwise,}
    \end{cases}
  \]
  defines a continuous t-norm on $[0,1]$; the right adjoint
  $\hom(x,-)$ of $x\otimes-$ is defined by
  \[
    \hom(x,y)=
    \begin{cases}
      1 & \text{if $x\le y$,}\\
      a_i+(b_i-a_i)\cdot\hom_i\left(\frac{x-a_i}{b_i-a_i},\frac{y-a_i}{b_i-a_i}\right) & \text{if $x,y\in[a_i,b_i]$,}\\
      y & \text{otherwise.}
    \end{cases}
  \]
\end{theorem}

In conclusion, the results of this section show that every continuous
t-norm on $[0,1]$ is obtained as a combination of infimum,
multiplication and \L{}ukasiewicz tensor.

\section{Stably compact spaces, partially ordered compact spaces, and
  Vietoris monads}\label{sec:stably-compact-Vietoris}

The notion of Vietoris monad plays a key role in the duality results
presented beginning from
Section~\ref{sec:duality-distributors}. Recall that, for a topological
space $X$, the \df{lower Vietoris space} of $X$ is the space
\[
  VX=\{A\subseteq X\mid A\text{ is closed}\}
\]
of closed subsets of $X$ with the topology generated by the sets
\[
  \{A\in VX\mid A\cap B\neq\varnothing\},
\]
where $B$ ranges over all open subsets $B\subseteq X$. Moreover, for a
continuous map $f:X\to Y$, the map
\[
  Vf:VX\longrightarrow VY,\,A\longmapsto \overline{f(A)}
\]
is continuous as well. These constructions define an endofunctor
$V:\TOP\to\TOP$ on the category $\TOP$ of topological spaces and
continuous maps that is part of the so-called \df{lower Vietoris
  monad} $\mV=\vmonad$ on $\TOP$; the unit $e$ and the multiplication
$m$ of $\mV$ have components
\begin{align*}
  e_X:X\longrightarrow VX,\,x\longmapsto\overline{\{x\}} &&\text{and}&& m_X:VVX\longrightarrow VX,\,\calA\longmapsto\bigcup\calA.
\end{align*}
For more information about this construction we refer to
\citep[Section 6.3]{Sch93}. 

Every topology on a set $X$ defines a \df{natural order relation} on
$X$: $x\le y$ if $x$ belongs to every neighborhood of $y$; in other
words, if the principal ultrafilter $\doo{x}$ converges to
$y$. Therefore we can use order-theoretic notions in topological
spaces. For instance, a subset $A$ of $X$ is called \df{lower}
whenever every $y\in X$ below some $x\in A$ belongs to $A$; the notion
of \df{upper} subset is defined dually.

\begin{example}
  For every topological space $X$, the underlying order of $VX$ is
  given by containment $\supseteq$.
\end{example}

A topological space $X$ is called \df{stably compact} (see
\citep{GHK+80,GHK+03}) whenever $X$ is sober, locally compact and
every finite intersection of compact lower subsets is compact. Stably
compact spaces are the objects of the category $\STCOMP$, whose
morphisms are the continuous maps with the property that the preimage
of a compact lower subset is compact. This kind of maps between stably
compact spaces are called \df{spectral}. We alert the reader that the
designation \emph{proper map} is also used in the literature for this
type of map (for instance in \citep{GHK+03}) which clashes with the
classical notion of proper mapping in topology (see \citep{Bou66}). It
is well-known (see \citep{Sch93}, for instance) that the monad $\mV$
on $\TOP$ restricts to a monad on $\STCOMP$, also denoted by
$\mV=\vmonad$.

There is a close connection between stably compact spaces and
partially ordered compact spaces which was first exposed in
\citep{GHK+80}; the notion of partially ordered compact space was
introduced in \citep{Nac50}. We recall that a partially ordered
compact space consists of a compact space $X$ equipped with an order
relation $\le$ so that
\[
  \{(x,y)\mid x\le y\}\subseteq X\times X
\]
is a closed subset of the product space $X \times X$. We denote by
$\POSCH$ the category of partially ordered compact spaces and monotone
continuous maps.

\begin{remark}
  It follows immediately from the definition that every partially
  ordered compact space is Hausdorff since the diagonal
  \[
    \Delta=\{(x,y)\mid x\leq y\}\cap\{(x,y)\mid y\leq x\}
  \]
  is a closed subset of $X\times X$.
\end{remark}

Given a partially ordered compact space $X$, keeping its topology but
taking now its dual order produces also a partially ordered compact
space, denoted by $X^\op$. Of particular interest to us are the
partially ordered compact space $[0,1]$ with the Euclidean topology
and the usual ``less or equal'' relation, and its dual partially ordered
compact space $[0,1]^\op$.

Below we collect some facts about these structures which can be found
in or follow from \citep[Proposition 4 and Theorems 1 and 4]{Nac65}.

\begin{proposition}
  \label{comp_closed}
  If $A$ is a compact subset of a partially ordered compact space $X$ then
  the sets
  \[
    \upc A=\{y\in X\mid y\ge x\text{ for some }x\in A\} \hspace{1em}\text{and}\hspace{1em}
    \downc A=\{y\in X\mid y\le x\text{ for some }x\in A\}
  \]
  are closed.
\end{proposition}

\begin{corollary}
  Let $A$ be a subset of a partially ordered compact space $X$. Then,
  $\ouc{A} \subseteq X$ is the smallest closed upper subset containing
  $A$.
\end{corollary}

\begin{proposition}\label{prop:orderUri}
  Let $A$ and $B$ be subsets of a partially ordered compact space $X$
  such that $A$ is a closed upper set, $B$ is a closed lower set and
  $A \cap B = \varnothing$. Then there exists a continuous and
  monotone function $\psi: X \to [0,1]$ such that $\psi(x)=1$ for every
  $x \in A$, and $ \psi(x)=0$ for every $x \in B$.
\end{proposition}

These results imply immediately:

\begin{proposition}\label{prop:initial-cogenerator}
  The partially ordered compact space $[0,1]$ is an initial
  cogenerator in $\POSCH$; that is, for every partially ordered
  compact space $X$, the cone $(\psi:X\to[0,1])_\psi$ of all morphisms
  from $X$ to $[0,1]$ is point-separating and initial with respect to
  the canonical forgetful functor $\POSCH\to\SET$ (see
  \citep{Tho09,HN15} for a description of initial cones in
  $\POSCH$). Since $[0,1]\simeq[0,1]^\op$ in $\POSCH$, also
  $[0,1]^\op$ is an initial cogenerator in $\POSCH$.
\end{proposition}

Using the results above, we are able to characterise epimorphisms and
regular monomorphisms in $\POSCH$.

\begin{proposition}
The regular monomorphisms in $\POSCH$ are precisely the embedding.
\end{proposition}
\begin{proof}
  Clearly, every regular monomorphism is an embedding. We show that
  the converse implication follows from \citep[Theorem 6]{Nac65}. Let
  $f:X\to Y$ be an embedding in $\POSCH$ and assume that $z\notin A$
  where $A=f[X]$. Consider
  \begin{align*}
    A_0&=A\cap\downc z, & A_1=A\cap\upc z,
  \end{align*}
  hence $A_0$ and $A_1$ are closed and every element of $A_0$ is
  strictly below every element of $A_1$. Therefore the map
  \[
    g:A_0\cup A_1 \to [0,1],\,x\mapsto
    \begin{cases}
      0 & \text{if $x\in A_0$,}\\
      1 & \text{if $x\in A_1$}
    \end{cases}
  \]
  is monotone and continuous. By \citep[Theorem 6]{Nac65}, $g$ can be
  extended to a continuous and monotone map $g:A\to[0,1]$, and with
  $g_0(z)=0$ and $g_1(z)=1$ extend $g$ to continuous and monotone maps
  $g_0,g_1:A\cup\{z\}\to[0,1]$. Applying \citep[Theorem 6]{Nac65}
  again yields a morphisms $g_0,g_1:Y\to[0,1]$, therefore we can
  construct morphisms $g_0,g_1:Y\to [0,1]$ with $g_0(z)\neq g_1(z)$
  and which coincided on the elements of $A$.
\end{proof}
\begin{corollary}
  The epimorphisms in $\POSCH$ are precisely the surjections.
\end{corollary}
\begin{proof}
  Clearly, every surjective morphism of $\POSCH$ is an epimorphism.
  Let $f:X\to Y$ be an epimorphism in $\POSCH$, we consider its
  factorisation $f=m\cdot e$ in $\POSCH$ with $e$ surjective and $m$
  an embedding. Hence, since $m$ is a regular monomorphism and an
  epimorphism, we conclude that $m$ is an isomorphism and therefore
  $f$ is surjective.
\end{proof}

The category $\POSCH$ is isomorphic to the category $\STCOMP$ of
stably compact spaces and spectral functions, for details we refer to
\citep{GHK+03}. Under this isomorphism, a partially ordered compact
space $X$ corresponds to the stably compact space with the same
underlying set and open sets precisely the open lower subsets of
$X$. In the reverse direction, a stably compact space $X$ defines a
partially ordered compact space whose order relation is the natural
order of $X$, and whose compact Hausdorff topology is the topology
generated by the open subsets and the compact lower subsets of $X$.
Therefore we can transfer the lower Vietoris monad on $\STCOMP$ to a
monad $\mV=\vmonad$ in $\POSCH$. Specifically, for a partially ordered
compact space $X$, the elements of $VX$ are the closed upper subsets
of $X$, the order on $VX$ is containment $\supseteq$, and the
sets
\begin{align*}
  \{A\in VX\mid A\cap U\neq\varnothing\}\hspace{1em}\text{($U\subseteq
  X$ open lower)} &&\text{and}&& \{A\in VX\mid A\cap
                                 K=\varnothing\}\hspace{1em}\text{($K\subseteq
                                 X$ closed lower)}.
\end{align*}
generated the associated compact Hausdorff topology. Furthermore, from
the monad $\mV$ on $\STCOMP$ we obtain a monad $\mV=\vmonad$ on the
category $\COMPHAUS$ of compact Hausdorff spaces and continuous maps
via the canonical adjunction
\[
  \POSCH\adjunct{\text{discrete}}{\text{forgetful}}\COMPHAUS.
\]
The functor $V:\COMPHAUS\to\COMPHAUS$ sends a compact Hausdorff space
$X$ to the space
\[
  V X=\{A\subseteq X\mid A\text{ is closed}\}
\]
with the topology generated by the sets
\begin{align*}
  \{A\in VX\mid A\cap U\neq\varnothing\}\hspace{1em}\text{($U\subseteq X$ open)} &&\text{and}&& \{A\in VX\mid A\cap K=\varnothing\}\hspace{1em}\text{($K\subseteq X$ closed)};
\end{align*}
for $f:X\to Y$ in $\COMPHAUS$, $V f:VX\to VY$ sends $A$ to
$f[A]$. We note that this is indeed the original construction
introduced by Vietoris in \citep{Vie22}. The unit $e$ and the
multiplication $m$ of $\mV$ are given by
\begin{align*}
 e_X:X\to VX,\,x\mapsto\{x\} &&\text{and}&& m_X:VV X\to V X,\,\calA\mapsto\bigcup\calA.
\end{align*}

In the next section we will be interested in the Kleisli categories
$\POSCH_{\mV}$ and $\COMPHAUS_{\mV}$. A morphism $X\to V Y$ in
$\COMPHAUS$ corresponds to a relation $X\relto Y$, and we refer to
those relations between compact Hausdorff spaces coming from morphisms
in $\COMPHAUS_{\mV}$ as \df{continuous relations}. Likewise, a
morphism $X\to VY$ in $\POSCH$ corresponds to a distributor between
the underlying partially ordered sets; we will call such a distributor
a \df{continuous distributor} of partially ordered compact
spaces. Furthermore, in both cases composition in the Kleisli category
corresponds to relational composition, therefore we will identify
$\COMPHAUS_{\mV}$ with the category $\COMPHAUSREL$ of compact
Hausdorff spaces and continuous relations, and $\POSCH_{\mV}$ with the
category $\POSCHDIST$ of partially ordered compact spaces and
continuous distributors. In the latter case, the identity morphism on
a partially ordered compact space is its order relation. Also note
that $\COMPHAUSREL$ is isomorphic to the full subcategory of
$\POSCHDIST$ determined by the discretely ordered compact spaces.

\section{Dual adjunctions}
\label{sec:dual-adjunctions}

In this section we present some well-known results about the structure
and construction of dual adjunctions. There is a vast literature on
this subject, we mention here \citep{LR78,LR79}, \citep{DT89},
\citep{PT91}, \citep{Joh86} and \citep{CD98}.

We start by considering an adjunction
\begin{equation}\label{eq:dual-adjunction}
  \catX\adjunctop{G}{F}\catA^\op
\end{equation}
between a category $\catX$ and the dual of a category $\catA$. In
general, such an adjunction is not an equivalence. Nevertheless, one
can always consider its restriction to the full subcategories
$\Fix(\eta)$ and $\Fix(\varepsilon)$ of $\catX$ respectively $\catA$, defined
by the classes of objects
\begin{align*}
  \{X\mid \eta_X\text{ is an isomorphism}\}
  &&\text{and}&&
                 \{A\mid \varepsilon_A\text{ is an isomorphism}\},
\end{align*}
where it yields an equivalence $\Fix(\eta)\simeq\Fix(\varepsilon)^\op$. The
passage from $\catX$ to $\Fix(\eta)$ is only useful if this
subcategory contains all ``interesting objects''. This, however, is
not always the case; $\Fix(\eta)$ can be even empty. \emph{En passant}
we mention that $\Fix(\eta)$ is a reflective subcategory of $\catX$
provided that $\eta G$ is an isomorphism; likewise, $\Fix(\varepsilon)$ is a
reflective subcategory of $\catA$ provided that $\varepsilon F$ is an
isomorphism. Moreover, $\eta G$ is an isomorphism if and only if
$\varepsilon F$ is an isomorphism (see \citep[Theorem 2.0]{LR79} for
details).

Throughout this section we assume that $\catX$ and $\catA$ are
equipped with faithful functors
\begin{align*}
  |-|:\catX \longrightarrow\SET &&\text{and}&& |-|:\catA \longrightarrow\SET.
\end{align*}
\begin{definition}
  The adjunction \eqref{eq:dual-adjunction} is \df{induced by the
    dualising object} $(\widetilde{X},\widetilde{A})$, with objects
  $\widetilde{X}$ in $\catX$ and $\widetilde{A}$ in $\catA$, whenever
  $|\widetilde{X}|=|\widetilde{A}|$, $|F|=\hom(-,\widetilde{X})$,
  $|G|=\hom(-,\widetilde{A})$ and the units are given by
  \begin{align}\label{eq:ev-maps}
    \eta_X:X &\longrightarrow GFX &\text{and}&& \varepsilon_A:A&\longrightarrow FGA;\\
    x&\longmapsto \ev_x &&& a&\longmapsto \ev_a\notag
  \end{align}
  with $\ev_x$ and $\ev_a$ denoting the evaluation maps.
\end{definition}
If the forgetful functors to $\SET$ are representable by objects $X_0$
in $\catX$ and $A_0$ in $\catA$, then every adjunction
\eqref{eq:dual-adjunction} is of this form, up to natural equivalence
(see \citep{DT89} and \citep{PT91}).
\begin{remark}\label{rem:diagrams_units}
  Consider an adjunction \eqref{eq:dual-adjunction} induced by a
  dualising object $(\widetilde{X},\widetilde{A})$. For every
  $\psi:X\to\widetilde{X}$ and $\varphi:A\to\widetilde{A}$, the
  diagrams
  \begin{align*}
    \xymatrix{X\ar[r]^-{\eta_X}\ar[dr]_\psi & GFX\ar[d]^{\ev_\psi}\\ & \widetilde{X}}
                                                                     &&\text{and}&&
                                                                                    \xymatrix{A\ar[r]^-{\varepsilon_A}\ar[dr]_\varphi & FGA\ar[d]^{\ev_\varphi}\\ & \widetilde{A}}
  \end{align*}
  commute.
\end{remark}

We turn now to the question ``How to construct dual
equivalences?''. Motivated by the considerations above, we assume that
$\widetilde{X}$ and $\widetilde{A}$ are objects in $\catX$ and $\catA$
respectively, with the same underlying set
$|\widetilde{X}|=|\widetilde{A}|$. In order to obtain a dual
adjunction, we wish to lift the hom-functors
$\hom(-,\widetilde{X}):\catX^\op\to\SET$ and
$\hom(-,\widetilde{A}):\catA^\op\to\SET$ to functors
$F:\catX^\op\to\catA$ and $G:\catA^\op\to\catX$ in such a way that the
maps \eqref{eq:ev-maps} underlie an $\catX$-morphism respectively and
$\catA$-morphism. To this end, we consider the following two
conditions.

\begin{description}[font=\normalfont,labelindent=\parindent]
\item[\namedlabel{InitX}{(Init X)}] For each object $X$ in $\catX$,
  the cone
  $(\ev_{x}:\hom(X,\widetilde{X})\to
  |\widetilde{A}|,\psi\mapsto\psi(x))_{x\in |X|}$ admits an initial
  lift $(\ev_{x}:F(X)\to\widetilde{A})_{x\in |X|}$.
\item[\namedlabel{InitA}{(Init A)}] For each object $A$ in $\catA$,
  the cone
  $(\ev_{a}:\hom(A,\widetilde{A})\to |\widetilde{X}|))_{a\in |A|}$
  admits an initial lift $(\ev_{a}:G(A)\to\widetilde{X})_{b\in V(B)}$.
\end{description}

\begin{theorem}\label{thm:dual-adjunction}
  If conditions \ref{InitX} and \ref{InitA} are fulfilled, then these
  initial lifts define the object parts of a dual adjunction
  \eqref{eq:dual-adjunction} induced by
  $(\widetilde{X},\widetilde{A})$.
\end{theorem}

Clearly, if $|-|:\catX\to\SET$ is topological (see \citep{AHS90}),
then \ref{InitX} is fulfilled. The following proposition describes
another typical situation.

\begin{proposition}\label{prop:init}
  Let $\catA$ be the category of algebras for a signature $\Omega$ of
  operation symbols and assume that $\catX$ is complete and
  $|-|:\catX\to\SET$ preserves limits. Furthermore, assume that, for
  every operation symbol $\omega\in\Omega$, the corresponding
  operation ${|\widetilde{A}|}^I\to|\widetilde{A}|$ underlies an
  $\catX$-morphism ${\widetilde{X}}^I\to\widetilde{X}$. Then both
  \ref{InitX} and \ref{InitA} are fulfilled.
\end{proposition}
\begin{proof}
  This result is essentially proven in \citep[Proposition
  2.4]{LR79}. Firstly, since all operations on $\widetilde{A}$ are
  $\catX$-morphisms, the algebra structure on $\hom(X,\widetilde{X})$
  can be defined pointwise. Secondly, for each algebra $A$, the
  canonical inclusion $\hom(A,\widetilde{A})\to |\widetilde{X}|^{|A|}$
  is the equaliser of a pair of $\catX$-morphisms between powers of
  $\widetilde{X}$. In fact, a map $f:|A|\to|\widetilde{A}|$ is an
  algebra homomorphism whenever, for every operation symbol
  $\omega\in\Omega$ with arity $I$ and every $h\in{|A|}^I$,
  \[
    f(\omega_A(h))=\omega_{\widetilde{A}}(f\cdot h).
  \]
  In other words, the set of maps $f:|A|\to|\widetilde{A}|$ which
  preserve the operation $\omega$ is precisely the equaliser of
  \[
    \pi_{\omega_A(h)}:{|\widetilde{A}|}^{|A|}\to |\widetilde{A}|
  \]
  and the composite
  \[
    {|\widetilde{A}|}^{|A|}\xrightarrow{\;-\cdot
      h\;}{|\widetilde{A}|}^I\xrightarrow{\;\omega_{\widetilde{A}}\;}|\widetilde{A}|.
  \]
  Since both maps underlie $\catX$-morphisms
  $\widetilde{X}^{|A|}\to\widetilde{X}$, the assertion follows.
\end{proof}

\begin{remark}\label{rem:init-lax}
  The result above remains valid if
  \begin{itemize}
  \item the objects of $\catA$ admit an order relation and some of the
    operations are only required to be preserved laxly, and
  \item the order relation $R\to|\widetilde{A}|\times|\widetilde{A}|$
    of $\widetilde{A}$ underlies an $X$-morphism
    $R'\to\widetilde{X}\times\widetilde{X}$.
  \end{itemize}
  In fact, with the notation of the proof above, the set of maps
  $f:|A|\to|\widetilde{A}|$ with
  \[
    f(\omega_A(h))\le\omega_{\widetilde{A}}(f\cdot h)
  \]
  for all $h\in{|A|}^I$ can be described as the pullback of the
  diagram
  \[
    \xymatrix{ & R\ar[d]\\ {|\widetilde{A}|}^{|A|}\ar[r] &
      |\widetilde{A}|\times|\widetilde{A}|.}
  \]
\end{remark}

Clearly, for every object $X$ in $\catX$, the unit $\eta_X:X\to GF(X)$
is an isomorphism if and only if $\eta_X$ is surjective and an
embedding. If the dual adjunction is constructed using \ref{InitX} and
\ref{InitA}, then, by Remark~\ref{rem:diagrams_units},
\begin{center}
  $\eta_X$ is an embedding if and only if the cone
  $(\psi:X\to\widetilde{X})_\psi$ is point-separating and initial.
\end{center}

We hasten to remark that the latter condition only depends on
$\widetilde{X}$ and is independent of the choice of $\catA$. If $\eta$
is not componentwise an embedding, we can substitute $\catX$ by its
full subcategory defined by all those objects $X$ where
$(\psi:X\to\widetilde{X})_\psi$ is point-separating and initial; by
construction, the functor $G$ corestricts to this subcategory. Again,
this procedure is only useful if this subcategory contains all
``interesting spaces'', otherwise it is probably best to use a
different dualising object. For exactly this reason, in this paper we
will consider the compact Hausdoff space $[0,1]$ instead of the
discrete two-element space.

We assume now that $\eta$ is componentwise an embedding. Then the
functor $F:\catX\to\catA^\op$ is faithful, and $\eta$ is an
isomorphism if and only if $F$ is also full. Put differently, if
$\eta$ is not an isomorphism, then $\catA$ has too many arrows. A
possible way to overcome this problem is to enrich the structure of
$\catA$. For instance, in \citep[VI.4.4]{Joh86} it is shown that,
under mild conditions, $\catA$ can be substituted by the category of
Eilenberg--Moore algebras for the monad on $\catA$ induced by the dual
of the adjunction \eqref{eq:dual-adjunction}. However, in this paper
we take a different approach: instead of saying ``$\catA$ has too many
morphisms'', one might also think that ``$\catX$ has too few
morphisms''. One way of adding morphisms to a category is passing from
$\catX$ to the Kleisli category $\catX_\mT$, for a suitable monad
$\mT$ on $X$. In fact, and rather trivially, for $\mT$ being the monad
on $\catX$ induced by the adjunction \eqref{eq:dual-adjunction}, the
comparison functor $\catX_\mT\to\catA^\op$ is fully faithful. In
general, this procedure will not give us new insights since we do not
know much about the monad induced by $F\dashv G$. The situation
improves if we take a different, better known monad $\mT$ on $\catX$
\emph{isomorphic} to the monad induced by $F\dashv G$. We are then
left with the task of identifying the $\catX$-morphisms inside
$\catX_\mT$ in a purely categorical way so that it can be translated
across a duality.

\begin{example}
  Consider the power monad $\mP$ on $\SET$ whose Kleisli category
  $\SET_\mP$ is equivalent to the category $\REL$ of sets and
  relations. Within $\REL$, functions can be identified by two
  fundamentally different properties.
  \begin{itemize}
  \item A relation $r:X\relto Y$ is a function if and only if $r$ has
    a right adjoint in the ordered category $\REL$. This is actually a
    2-categorical property; if we want to use it in a duality we must
    make sure that the involved equivalence functors are locally
    monotone.
  \item A relation $r:X\relto Y$ is a function if and only if the
    diagrams
    \begin{align*}
      \xymatrix{X\ar[r]|-{\object@{|}}^r\ar[rd]|-{\object@{|}}_\top & Y\ar[d]|-{\object@{|}}^\top \\ &1}
                                                                                                     &&\text{and}&&
                                                                                                                    \xymatrix{X\times X\ar[r]|-{\object@{|}}^{r\times r} & Y\times Y\\
      X\ar[u]|-{\object@{|}}^\Delta\ar[r]|-{\object@{|}}_r & Y\ar[u]|-{\object@{|}}_\Delta}
    \end{align*}
    commute. In the second diagram, $X\times X$ denotes the
    set-theoretical product which can be misleading since it is not
    the categorical product in $\REL$. To use this description in a
    duality result, one needs to know the corresponding operation on
    the other side.
  \end{itemize}
\end{example}

In the considerations above, the Kleisli category $\catX_\mT$ was only
introduced to support the study of $\catX$; however, at some occasions
our primary interest lies in $\catX_\mT$. In this case, a monad $\mT$
on $\catX$ is typically given before-hand, and we wish to find an
adjunction \eqref{eq:dual-adjunction} so that the induced monad is
isomorphic to $\mT$.  If a dualising object
$(\widetilde{X},\widetilde{A})$ induces this adjunction, we speak of a
\df{functional representation} of $\mT$.  Looking again at the example
$\CoAlg(V)\simeq\BAO^\op$ of Section~\ref{sec:intro}, by observing
that $V$ is part of a monad $\mV=\vmonad$ on $\STONE$, we can
think of the objects of $\CoAlg(V)$ as Boolean spaces $X$ equipped
with an endomorphism $r:X\relto X$ in $\STONE_{\mV}$; the morphisms
of $\CoAlg(V)$ are those morphisms of $\STONE$ commuting with this
additional structure. Halmos' duality theorem \citep{Hal56} affirms
that the category $\STONE_{\mV}$ is dually equivalent to the category
$\SLAT_\bool$ of Boolean algebras with $\vee$-semilattice
homomorphims, that is, maps preserving finite suprema but not
necessarily finite infima.  The duality $\CoAlg(V)\simeq\BAO^\op$
follows now from both Halmos duality and the classical Stone duality
$\STONE\simeq\BOOL^\op$ \citep{Sto36}. We note that \citep{Hal56} does
not talk about monads, but in \citep{HN15} we have studied this and
other dualities from this point of view.

As we explained above, our aim is to construct and analyse functors
$F:\catX_{\mT}\to\catA^\op$ which extend a given functor
$F:\catX\to\catA^\op$ that is part of an adjunction $F\dashv G$
induced by a dualising object $(\widetilde{X},\widetilde{A})$. It is
well-known that such functors $F:\catX_\mT\to\catA^\op$ correspond
precisely to monad morphisms from $\mT$ to the monad induced by
$F\dashv G$, and that monad morphisms into a ``double dualisation
monad'' are in bijection with certain algebra structures on
$\widetilde{X}$ (see \citep{Koc71b}, for instance). In the remainder
of this section, we explain these correspondences in the specific
context of our paper.

Let $\catX$ and $\catA$ be categories with respresentable forgetful
functors
\begin{align*}
  |-|\simeq\hom(X_0,-):\catX\to\SET &&\text{and}&& |-|\simeq\hom(A_0,-):\catA\to\SET,
\end{align*}
$\mT=\monad$ a monad on $\catX$ and $F\dashv G$ an adjunction
\[
  \catX\adjunctop{G}{F}\catA^\op
\]
induced by $(\widetilde{X},\widetilde{A})$. We denote by $\mD$ the
monad induced by $F\dashv G$. The next result establishes a connection
between monad morphisms $j:\mT\to\mD$ and $\mT$-algebra structures on
$\widetilde{X}$ \emph{compatible} with the adjunction $F\dashv G$.

\begin{theorem}\label{thm:Kleisli-adjunction}
  In the setting described above, the following data are in bijection.
  \begin{enumerate}
  \item\label{item:1} Monad morphisms $j:\mT\to\mD$.
  \item\label{item:2} Functors $F:\catX_{\mT}\to\catA^\op$ making the
    diagram
    \[
      \xymatrix{\catX_{\mT}\ar[r]^F & \catA^\op\\
        \catX\ar[u]^{F_{\mT}}\ar[ur]_F}
    \]
    commutative.
  \item\label{item:3} $\mT$-algebra structures
    $\sigma:T\widetilde{X}\to\widetilde{X}$ such that the map
    \[
      \hom(X,\widetilde{X})\longrightarrow\hom(TX,\widetilde{X}),\,\psi\longmapsto\sigma\cdot
      T\psi
    \]
    is an $\catA$-morphism $\kappa_X:FX\to FTX$, for every object $X$
    in $\catX$.
  \end{enumerate}
\end{theorem}
\begin{proof}
  The equivalence between the data described in \eqref{item:1} and
  \eqref{item:2} is well-known, see \citep{Pum70}, for instance. We
  recall here that, for a monad morphism $j:\mT\to\mD$, the
  corresponding functor $F:\catX_{\mT}\to\catA^\op$ can be obtained as
  \[
    \catX_{\mT}\xrightarrow{\text{composition with
        $j$}}\catX_\mD\xrightarrow{\text{comparison}}\catA^\op.
  \]
  To describe the passage from (\ref{item:1}) to (\ref{item:3}), we
  recall from \citep[Lemma VI.4.4]{Joh86} that $\widetilde{X}$ becomes
  a $\mD$-algebra since $\widetilde{X}\simeq GA_0$ and
  $G:\catA^\op\to\catX$ factors as
  \[
    \xymatrix{\catA^\op\ar[drr]_{G}\ar[rr]^{\text{comparision}} &&
      \catX^\mD\ar[d]^{\text{forgetful}} \\ && \catX.}
  \]
  A little computation shows that the $\mD$-algebra structure on
  $\widetilde{X}$ is
  
  \[
    GF\widetilde{X}\xrightarrow{\ev_{1_{\widetilde{X}}}}\widetilde{X}.
  \]
  Composing $\ev_{1_{\widetilde{X}}}$ with $j_{\widetilde{X}}$ gives a
  $\mT$-algebra structure
  $\sigma:T\widetilde{X}\to\widetilde{X}$. Furthermore, the functor
  $F:\catX_{\mT}\to\catA^\op$ sends $1_{TX}:TX\relto X$ in $\catX_\mT$
  to the $\catA$-morphism $Fj_X\cdot\varepsilon_{FX}:FX\to FTX$ which sends
  $\psi\in FX$ to $Fj_X(\ev_\psi)=\ev_\psi\cdot j_X$. On the other
  hand,
  \[
    \sigma\cdot T\psi=\ev_{1_{\widetilde{X}}}\cdot
    j_{\widetilde{X}}\cdot T\psi=\ev_{1_{\widetilde{X}}}\cdot
    GF\psi\cdot j_X=\ev_\psi\cdot j_X;
  \]
  which shows that $\kappa_X=Fj_X\cdot\varepsilon_{FX}$ is an
  $\catA$-morphism. For a compatible $\mT$-algebra structures
  $\sigma:T\widetilde{X}\to\widetilde{X}$ as in (\ref{item:3}),
  \[
    (\varphi:X\to TY)\longmapsto (FY\xrightarrow{\kappa_Y}
    FTY\xrightarrow{F\varphi}FX)
  \]
  defines a functor $F:\catX_{\mT}\to\catA^\op$ making the diagram
  \[
    \xymatrix{\catX_{\mT}\ar[r]^F & \catA^\op\\
      \catX\ar[u]^{F_{\mT}}\ar[ur]_F}
  \]
  commutative. The induced monad morphism $j:\mT\to\mD$ is given by
  the family of maps
  \[
    j_X:|TX|\longrightarrow\hom(FX,\widetilde{A}),\,\fx\longmapsto(\psi\mapsto\sigma\cdot
    T\psi(\fx)).
  \]
  Furthermore, the $\mT$-algebra structure induced by this $j$ is
  indeed
  \[
    \ev_{1_{\widetilde{X}}}\cdot j_{\widetilde{X}}=\sigma\cdot
    T1_{\widetilde{X}}=\sigma.
  \]
  Finally, for a monad morphism $j:\mT\to\mD$, the monad morphism
  induced by the corresponding algebra structure $\sigma$ has as
  $X$-component the map sending $\fx\in TX$ to
  \[
    \sigma\cdot T\psi(\fx)=\ev_\psi\cdot
    j_X(\fx)=j_X(\fx)(\psi).\qedhere
  \]
\end{proof}

\begin{remark}\label{rem:Kleisli-adjunction}
  The constructions described above seem to be more natural if
  $\widetilde{X}=TX_0$ with $\mT$-algebra structure $m_{X_0}$, see
  \citep[Proposition 4.3]{HN15}. In this case, the functor
  $F:\catX_{\mT}\to\catA^\op$ is a lifting of the hom-functor
  $\hom(-,X_0):\catX_{\mT}\to\SET^\op$. Furthermore, interpreting the
  elements of $TX$ as morphisms $\varphi:X_0\relto X$ in the Kleisli
  category $\catX_\mT$ allows to describe the components of the monad
  morphism $j$ using composition in $\catX_\mT$:
  \[
    j_X:|TX|\longrightarrow\hom(FX,\widetilde{A}),\,\varphi\longmapsto(\psi\mapsto\psi\cdot\varphi).
  \]
  In Section~\ref{sec:enriched-Vietoris} we apply this construction
  to a variation of the Vietoris monad on a category of ``metric
  compact Hausdoff spaces''. Unlike the classical Vietoris functor,
  the functor of this monad sends the one-element space to
  $[0,1]^\op$.
\end{remark}

\section{Duality theory for continuous
  distributors}\label{sec:duality-distributors}

In this section we apply the results presented in
Section~\ref{sec:dual-adjunctions} to the Vietoris monad $\mV$ on
$\catX=\POSCH$, with $\widetilde{X}=[0,1]^\op$ and $\mV$-algebra
structure
\[
  V([0,1]^\op)\longrightarrow [0,1]^\op,\,A\longmapsto\sup_{x\in A} x.
\]
Then, for a category $\catA$ and an adjunction
\begin{equation*}
  \POSCH\adjunctop{G}{C}\catA^\op
\end{equation*}
induced by $([0,1]^\op,[0,1])$ and compatible with the $\mV$-algebra
structure on $[0,1]^\op$, the corresponding monad morphism $j$ has as
components the maps
\[
  j_X:VX\longrightarrow
  GC(X),\,A\longmapsto(\Phi_A:CX\to[0,1],\,\psi\mapsto\sup_{x\in
    A}\psi(x)).
\]
We wish to find an appropriate category $\catA$ so that $j$ is an
isomorphism. Our first inspiration stems from \citep{Sha92} where the
following result is proven.

\begin{theorem}
  Consider the subfunctor $V_1:\COMPHAUS\to\COMPHAUS$ of $V$
  sending $X$ to the space of all non-empty closed subsets of $X$. The
  functor $V_1:\COMPHAUS\to\COMPHAUS$ is naturally isomorphic to the
  functor which sends $X$ to the space of all functions
  \[
    \Phi:C(X,\R_0^+)\to\R_0^+
  \]
  satisfying the conditions (for all
  $\psi,\psi_1,\psi_2\in C(X,\R_0^+)$ and $u\in\R_0^+$)
  \begin{enumerate}
  \item $\Phi$ is monotone,
  \item\label{cond.act} $\Phi(u*\psi)=u*\Phi(\psi)$,
  \item\label{cond.add}
    $\Phi(\psi_1 + \psi_2)\le\Phi(\psi_1)+\Phi(\psi_2)$,
  \item $\Phi(\psi_1 \cdot \psi_2)\le\Phi(\psi_1)\cdot\Phi(\psi_2)$,
  \item\label{cond.addconst} $\Phi(\psi_1 + u)=\Phi(\psi_1)+u$,
  \item\label{cond.const} $\Phi(u)=u$.
  \end{enumerate}
  The topology on the set of all maps $\Phi:C(X,\R_0^+)\to\R_0^+$
  satisfying the conditions above is the initial one with respect to
  all evaluation maps $\ev_\psi$, where $\psi\in C(X,\R_0^+)$. The
  $X$-component of the natural isomorphism sends a closed non-empty
  subset $A\subseteq X$ to the map $\Phi_A:C(X,\R_0^+)\to\R_0^+$
  defined by
  \[
    \Phi_A(\psi)=\sup_{x\in A}\psi(x).
  \]
\end{theorem}

One notices immediately that if we allow $A=\varnothing$, $\Phi_A$
does not satisfy the last two axioms above.  In fact, as we show
below, the condition \eqref{cond.addconst} is not necessary for
\citeauthor{Sha92}'s result; moreover, when generalising from
multiplication $*$ to an arbitrary continuous quantale structure
$\otimes$ on $[0,1]$, we change $+$ in \eqref{cond.addconst} to
truncated minus, which is compatible with the empty set. Thanks to
\eqref{cond.act}, the condition~\ref{cond.const} can be equivalently
expressed as $\Phi(1)=1$, and this is purely related to
$A\neq\varnothing$ (see Proposition~\ref{prop:functional-relation}).

\begin{remark}
  There is also interesting work of \citeauthor{Rad97} on a
  ``functional representation of the Vietoris monad'' in terms of
  functionals, notably \citep{Rad97,Rad09}. In particular,
  \citeauthor{Rad97} shows that the Vietoris monad is isomorphic to
  the monad defined by all real-valued ``functionals which are normed,
  weakly additive, preserve $\max$ and weakly preserve $\min$''.
\end{remark}

To fit better into our framework, in the sequel we will consider
functions into $[0,1]$ instead of $\R_0^+$, and also consider binary
suprema $\vee$ instead of $+$ in \eqref{cond.add}.

\begin{assumption}\label{ass:general-assumption}
  From now on $\otimes$ is a quantale structure on $[0,1]$ with
  neutral element $1$. Note that then necessarily
  $u\otimes v\le u\wedge v$, for all $u,v\in[0,1]$. In order to be
  able to combine continuous functions $\psi_1,\psi_2:X\to[0,1]$, we
  assume that $\otimes:[0,1]\times[0,1]\to[0,1]$ is continuous with
  respect to the Euclidean topology on $[0,1]$. In other words, we
  consider a continuous t-norm on $[0,1]$.
\end{assumption}

We recall that $\FinSup{[0,1]}$ denotes the category of separated
finitely cocomplete $[0,1]$-categories and finite colimit preserving
$[0,1]$-functors; the unit interval $[0,1]$ equipped with
$\hom:[0,1]\times[0,1]\to[0,1]$ is certainly an object of
$\FinSup{[0,1]}$. We introduce now the following categories.

\begin{itemize}
\item The category
  \[
    \Mnd(\FinSup{[0,1]})
  \]
  of commutative monoids and monoid homomorphisms in
  $\FinCoCts{[0,1]}$; that is, the objects of $\Mnd(\FinCoCts{[0,1]})$
  are separated finitely cocomplete $[0,1]$-categories $X$ equipped
  with an associative and commutative operation $\ocircle:X\times X\to X$
  with unit element $1$ the top-element of $X$ and the morphisms of
  $\Mnd(\FinCoCts{[0,1]})$ are finitely cocontinuous $[0,1]$-functors
  preserving the unit and the multiplication. We also assume that
  $x\ocircle -:X\to X$ is a finitely cocontinuous $[0,1]$-functor, for every
  $x\in X$.
\item The category
  \[
    \LaxMnd(\FinSup{[0,1]})
  \]
  of commutative monoids and lax monoid homomorphisms in
  $\FinSup{[0,1]}$; that is, $\LaxMnd(\FinSup{[0,1]}$ has the same
  objects as $\LaxMnd(\FinSup{[0,1]})$ and the morphisms are finitely
  cocontinuous $[0,1]$-functors $f:X\to Y$ satisfying
  \[
    f(x\ocircle x')\le f(x)\ocircle f(x'),
  \]
  for all $x,x'\in X$.
\end{itemize}

Note that, for every $u\in[0,1]$ and $x\in X$, we have
\[
  x\ocircle(1\otimes u)=(x\ocircle 1)\otimes u=x\otimes u,
\]
therefore we can think of $\ocircle:X\times X\to X$ as an ``extension'' of
$\otimes:X\times[0,1]\to X$ and write $x\otimes x'$ instead of
$x\ocircle x'$. Thinking more in algebraic terms, $\Mnd(\FinSup{[0,1]})$ is a
$\aleph_1$-ary quasivariety; in fact, by adding to the algebraic
theory of $\FinSup{[0,1]}$ (see
Remark~\ref{rem:FinCoCts_quasivariety}) the operations and equations
describing the monoid structure, one obtains a presentation by
operations and implications. It follows in particular that
$\Mnd(\FinSup{[0,1]})$ is complete and cocomplete. The forgetful
functor
\[
  \Mnd(\FinSup{[0,1]})\to\FinSup{[0,1]}
\]
preserves limits, and it is easy to see that the faithful functor
\[
  \Mnd(\FinSup{[0,1]})\to\LaxMnd(\FinSup{[0,1]})
\]
preserves limits as well.

In the sequel we consider the $[0,1]$-category $[0,1]$ as an object of
$\Mnd(\FinSup{[0,1]})$ with multiplication
$\otimes:[0,1]\times[0,1]\to[0,1]$. Note that
$\otimes:[0,1]^\op\times[0,1]^\op\to[0,1]^\op$ is a morphism in
$\POSCH$. From Theorem~\ref{thm:dual-adjunction} and
Remark~\ref{rem:init-lax}, we obtain:

\begin{proposition}\label{prop:induced-dual-adjunction}
  The dualising object $([0,1]^\op,[0,1])$ induces a natural dual
  adjunction
  \[
    \POSCH\adjunctop{G}{C}\LaxMnd(\FinSup{[0,1]})^\op.
  \]
  Here $CX$ is given by $\POSCH(X,[0,1]^\op)$ with all operations
  defined pointwise, and $GA$ is the space
  $\LaxMnd(\FinSup{[0,1]})(A,[0,1])$ equipped with the initial
  topology with respect to all evaluation maps
  \[
    \ev_a:\LaxMnd(\FinSup{[0,1]})(A,[0,1])\longrightarrow[0,1]^\op,\,\Phi\longmapsto\Phi(a).
  \]
\end{proposition}
\begin{proof}
  In terms of the algebraic presentation of the $[0,1]$-category
  $[0,1]$ of Remark~\ref{rem:FinCoCts_quasivariety}, the operations
  $\vee$ and $-\otimes u$ are morphisms
  $\vee:[0,1]^\op\times[0,1]^\op\to[0,1]^\op$ and
  $-\otimes u:[0,1]^\op\to[0,1]^\op$ in $\POSCH$, and the order
  relation of $[0,1]^\op$ is closed in $[0,1]^\op\times
  [0,1]^\op$. Therefore the assertion follows from
  Theorem~\ref{thm:dual-adjunction} and Proposition~\ref{prop:init}.
\end{proof}

The partially ordered compact space $[0,1]^\op$ is a $\mV$-algebra
with algebra structure $\sup:V([0,1]^\op)\to[0,1]^\op$, and one easily
verifies that
\[
  \hom(X,[0,1]^\op)\longrightarrow\hom(VX,[0,1]^\op),\,\psi\longmapsto(A\mapsto\sup_{x\in
    A}\psi(x))
\]
is a morphism $CX\to CVX$ in $\LaxMnd(\FinSup{[0,1]})$. By
Theorem~\ref{thm:Kleisli-adjunction} and
Remark~\ref{rem:Kleisli-adjunction}, we obtain a commutative diagram
\[
  \xymatrix{\POSCHDIST\ar[rr]^{C} && \LaxMnd(\FinSup{[0,1]})^\op,\\
    & \POSCH\ar[ul]\ar[ur]_C}
\]
of functors; where, for $\varphi:X\modto Y$ in $\POSCHDIST$,
\begin{align*}
  C\varphi:CY &\longrightarrow CX\\
  \psi &\longmapsto \left(x\mapsto\sup_{x\,\varphi\, y}\psi(y)\right).
\end{align*}
The induced monad morphism $j$ is precisely given by the family of
maps
\[
  j_X:VX\longrightarrow\LaxMnd(\FinSup{[0,1]})(CX,[0,1]),\;A\longmapsto\Phi_A,
\]
with
\[
  \Phi_A:CX\longrightarrow[0,1],\;\psi\longmapsto\sup_{x\in A}\psi(x).
\]

In order to show that $j$ is an isomorphism, it will be convenient to
refer individually to the components of the structure of $CX$; that
is, we consider the following conditions on a map $\Phi:CX\to[0,1]$.

\begin{description}[font=\normalfont,labelindent=\parindent]
\item [\namedlabel{Mon}{(Mon)}] $\Phi$ is monotone.
\item [\namedlabel{Act}{(Act)}] For all $u\in[0,1]$ and $\psi\in CX$,
  $\Phi(u\otimes\psi)=u\otimes\Phi(\psi)$.
\item [\namedlabel{Sup}{(Sup)}] For all $\psi_1,\psi_2\in CX$,
  $\Phi(\psi_1\vee\psi_2)=\Phi(\psi_1)\vee\Phi(\psi_2)$.
\item [\namedlabel{LTen}{(Ten)$_\textrm{lax}$}] For all
  $\psi_1,\psi_2\in CX$,
  $\Phi(\psi_1\otimes \psi_2)\le\Phi(\psi_1)\otimes\Phi(\psi_2)$.
\item [\namedlabel{Ten}{(Ten)}] For all $\psi_1,\psi_2\in CX$,
  $\Phi(\psi_1\otimes \psi_2)=\Phi(\psi_1)\otimes\Phi(\psi_2)$.
\item [\namedlabel{Top}{(Top)}] $\Phi(1)=1$.
\end{description}

\begin{remarks}
  The condition \Act{} implies $\Phi(0)=0$ and \Sup{} and implies
  \Mon. Also note that, by \Mon{} and \Act, if $\psi(x)\le u$ for all
  $x\in X$, then $\Phi(\psi)\le u$. Finally, if $\otimes=\wedge$, then
  \LTen{} is a consequence of \Mon.
\end{remarks}

Clearly, $\Phi:CX\to[0,1]$ is a morphism in $\LaxMnd(\FinSup{[0,1]})$
if and only if $\Phi$ satisfies the conditions \Mon, \Act, \Sup{} and
\LTen; and $\Phi:CX\to[0,1]$ is a morphism in $\Mnd(\FinSup{[0,1]})$
if and only if $\Phi$ satisfies the conditions\Mon, \Act, \Sup, \Top{}
and \Ten.

\begin{definition}\label{d:def:1}
  A closed upper subset $A\subseteq X$ of a partially ordered compact
  space is called \df{irreducible} whenever, for all closed upper
  subsets $A_1,A_2\subseteq X$ with $A_1\cup A_2=A$, one has $A=A_1$
  or $A=A_2$.
\end{definition}
With this definition, the empty set $\varnothing$ is irreducible; by
soberness of the corresponding stably compact space, the non-empty
irreducible closed subset $A\subseteq X$ are precisely the subsets of
the form $A=\upc x$, for some $x\in X$.

\begin{proposition}\label{prop:functional-relation}
  Let $X$ be a partially ordered compact space and $A\subseteq X$ a
  closed upper subset of $X$. Then the following assertions hold.
  \begin{enumerate}
  \item\label{item:5} $A\neq\varnothing$ if and only if $\Phi_A$ satisfies \Top.
  \item\label{item:6} $A$ is irreducible if and only if $\Phi_A$ satisfies \Ten.
  \end{enumerate}
\end{proposition}
\begin{proof}
  (\ref{item:5}) is clear, and so is the implication ``$\implies$'' in
  (\ref{item:6}). Assume now that $\Phi_A$ satisfies \Ten{} and let
  $A_1,A_2\subseteq X$ be closed upper subsets with $A_1\cup
  A_2=A$. Let $x\notin A_1$ and $y\notin A_2$. We find
  $\psi_1,\psi_2\in CX$ with
  \begin{align*}
    \psi_1(x)&=1, & \psi_2(y)&=1, & \forall z\in A\,.\,\psi_1(z)=0\text{ or }\psi_2(z)=0.
  \end{align*}
  Therefore
  \[
    0=\Phi_A(\psi_1\otimes\psi_2)=\Phi_A(\psi_1)\otimes\Phi_A(\psi_2).
  \]
  By Corollary~\ref{cor:no-zero-div}, $\Phi_A(\psi_1)=0$ or, for some
  $n\in\N$, $\Phi_A(\psi_2^n)=\Phi_A(\psi_2)^n=0$, hence $x\notin A$
  or $y\notin A$. We conclude that $A=A_1$ or $A=A_2$.
\end{proof}

\begin{corollary}\label{cor:restriction-to-functions}
  Let $\varphi:X\modto Y$ in $\POSCHDIST$. Then:
  \begin{enumerate}
  \item $\varphi$ is a total relation if and only of $C\varphi$
    preserves $1$.
  \item $\varphi$ is a partial function (deterministic relation in
    \citep[Definition 3.3.3]{Keg99}) if and only if $C\varphi$
    preserves $\otimes$.
  \end{enumerate}
\end{corollary}

Our next goal is to invert the process $A\mapsto\Phi_A$. Firstly,
following \citep{Sha92}, we introduce the subsequent notation.
\begin{itemize}
\item For every map $\psi:X\to[0,1]$,
  $\Zero(\psi)=\{x\in X\mid \psi(x)=0\}$ denotes the zero-set of
  $\psi$. If $\psi$ is a monotone and continuous map
  $\psi:X\to[0,1]^\op$, then $\Zero(\psi)$ is an closed upper subset
  of $X$.
\item For every map $\Phi:CX\to[0,1]$, we put
  \[
    \Zero(\Phi)=\bigcap\{\Zero(\psi)\mid\psi\in
    CX,\,\Phi(\psi)=0\}\subseteq X.
  \]
  Note that $\Zero(\Phi)$ is a closed upper subset of $X$.
\end{itemize}

There is arguably a more natural candidate for an inverse of
$j_X$. First note that, given a set $\{A_i\mid i\in I\}$ of closed
upper subsets of $X$ with $A=\overline{\bigcup_{i\in I}A_i}$, for
every $\psi\in CX$ one verifies
\[
  \Phi_A(\psi) =\sup_{x\in\bigcup_{i\in I}A_i}\psi(x)
  =\sup_{i\in I}\Phi_{A_i}(\psi).
\]
Hence, the monotone map $j_X$ preserves infima\footnote{Note that the
  order is reversed.} and therefore has a left adjoint which sends a
morphism $\Phi:CX\to[0,1]$ to
\[
  \Anti(\Phi)=\bigcap_{\psi\in CX}\psi^{-1}[0,\Phi(\psi)].
\]
In the sequel it will be convenient to consider the maps $\Zero$ and
$\Anti$ defined on the set $\{\Phi:CX\to[0,1]\}$ of all maps from $CX$
to $[0,1]$. We have the following elementary properties. 

\begin{lemma}
  Let $X$ be a partially ordered compact space $X$.
  \begin{enumerate}
  \item The maps $\Anti,\Zero:\{\Phi:CX\to[0,1]\}\to VX$ are monotone.
  \item $\Anti(\Phi)\subseteq\Zero(\Phi)$, for every map
    $\Phi:CX\to[0,1]$.
  \item For every $A\in VX$,
    $\Zero\cdot j_X(A))=A=\Anti\cdot j_X(A))$.
  \item For every map $\Phi:CX\to[0,1]$ and every $\psi\in CX$,
    $j_X\cdot\Anti(\Phi)(\psi)\le\Phi(\psi)$.
  \end{enumerate}
\end{lemma}

\begin{corollary}
  For every partially ordered compact space $X$, the map
  $j_X:VX\to GCX$ is an order-embedding.
\end{corollary}

Now, we wish to give conditions on $\Phi:CX\to[0,1]$ so that $j_X$
restricts to a bijection between $VX$ and the subset of
$\{\Phi:CX\to[0,1]\}$ defined by these conditions. In particular, we
consider:
\begin{enumerate}[label=(\Alph*)]
\item\label{CondA} For all $x\in X$ and all $\psi\in CX$, if
  $\psi(x)>\Phi(\psi)=0$, then there exists some $\bar{\psi}\in CX$
  with $\bar{\psi}(x)=1$ and $\Phi(\bar{\psi})=0$.
\end{enumerate}

\begin{lemma}
  Let $X$ be a partially ordered compact space.
  \begin{enumerate}
  \item If $\Phi:CX\to[0,1]$ satisfies \Mon, \Act{} and \LTen, then
    $\Phi$ satisfies~\ref{CondA}.
  \item If the quantale $[0,1]$ does not have nilpotent elements and
    $\Phi:CX\to[0,1]$ satisfies \Mon{} and \Act, then $\Phi$
    satisfies~\ref{CondA}.
  \end{enumerate}
\end{lemma}
\begin{proof}
  Assume $\psi(x)>\Phi(\psi)=0$. Put $v=\psi(x)$ and take $u$ with
  $0<u<v$. Put $A=\psi^{-1}([0,u])$. By
  Proposition~\ref{prop:orderUri}, there is some $\bar{\psi}\in CX$
  with $A\subseteq \Zero(\bar{\psi})$ and
  $\bar{\psi}(x)=1$. Furthermore,
  \[
    u\otimes\bar{\psi}\le u\wedge\bar{\psi}\le\psi
  \]
  and therefore $u\otimes\Phi(\bar{\psi})\le\Phi(\psi)=0$. Since
  $u\neq 0$, we get $\Phi(\bar{\psi})^n=0$ for some $n\in\N$. If there
  are no nilpotent elements, then $\Phi(\bar{\psi})=0$. In general,
  using condition \LTen{} we obtain $\Phi(\psi^n)\le\Phi(\psi')^n=0$
  and $\psi^n(x)=1$.
\end{proof}

Inspired by \citeauthor{Sha92}'s proof we get the following result.

\begin{proposition}\label{prop:Phi-less-or-equal}
  Let $X$ be a partially ordered compact space. For every
  $\Phi:CX\to[0,1]$ satisfying \Mon, \Act, \Sup{} and \CondA,
  \[
    \Phi(\psi)\le j_X\cdot \Zero(\Phi)(\psi),
  \]
  for all $\psi\in CX$.
\end{proposition}
\begin{proof}
  Let $\psi\in CX$, we wish to show that
  $\Phi(\psi)\le\sup_{x\in Z(\Phi)}\psi(x)$. To this end, consider an
  element $u\in[0,1]$ with $\sup_{x\in Z(\Phi)}\psi(x)<u$. Put
  \[
    U=\{x\in X\mid\psi(x)<u\}.
  \]
  Clearly, $U$ is open and $Z(\Phi)\subseteq U$. Let now
  $x\in X\setminus Z(\Phi)$. There is some $\psi'\in CX$ with
  $\Phi(\psi')=0$ and $\psi'(x)\neq 0$; by \CondA{} we may assume
  $\psi'(x)=1$. Let now $\alpha<1$. For every $\psi'\in CX$ we put
  \[
    \supp_\alpha(\psi')=\{x\in X\mid\psi'(x)>\alpha\}.
  \]
  By the considerations above,
  \[
    X=U\cup\bigcup\{\supp_\alpha(\psi')\mid\psi'\in
    C(X),\Phi(\psi')=0\};
  \]
  since $X$ is compact, we find $\psi_1,\dots,\psi_n$ with
  $\Phi(\psi_i)=0$ and
  \[
    X=U\cup\supp_\alpha(\psi_1)\cup\dots\cup\supp_\alpha(\psi_n).
  \]
  Hence,
  \[
    \alpha\otimes\psi\le
    u\vee(\psi_1\otimes\psi)\vee\dots\vee(\psi_n\otimes\psi),
  \]
  and therefore
  \[
    \alpha\otimes\Phi(\psi) \le
    u\vee\Phi(\psi_1\otimes\psi)\vee\dots\vee\Phi(\psi_n\otimes\psi)
    \le u\vee\Phi(\psi_1)\vee\dots\Phi(\psi_n) =u.\qedhere
  \]
\end{proof}

Hence, under the conditions of the proposition above, we have
\[
  \sup_{x\in\Anti(\Phi)}\psi(x)\le\Phi(\psi)\le\sup_{x\in\Zero(\Phi)}\psi(x),
\]
for all $\psi\in CX$. We investigate now conditions on
$\Phi:CX\to[0,1]$ which guarantee $\Zero(\Phi)=\Anti(\Phi)$.

\begin{proposition}
  Assume $\otimes=*$ or $\otimes=\luk$. If $\Phi$ satisfies
  \Mon{}, \Act{} and \LTen{} then $\Zero(\Phi)=\Anti(\Phi)$.
\end{proposition}
\begin{proof}
  We consider first $\otimes=*$, in this case the proof is essentially
  taken from \citep{Sha92}. For every $\psi\in CX$ and every open
  lower subset $U\subseteq X$ with $U\cap Z(\Phi)\neq\varnothing$, we
  show that $\inf_{x\in U}\psi(x)\le\Phi(\psi)$. To see this, put
  $u=\inf_{x\in U}\psi(x)$. Since there exists $z\in U\cap Z(\Phi)$,
  there is some $\psi'\in CX$ with $U^\complement\subseteq Z(\psi')$
  and $\psi'(z)=1$; thus $\Phi(\psi')\neq 0$. Then
  $u*\psi'\le\psi*\psi'$ and therefore
  $u*\Phi(\psi')\le\Phi(\psi)*\Phi(\psi')$. Since $\Phi(\psi')\neq 0$,
  we obtain $u\le\Phi(\psi)$.

  Let $x\in\Zero(\Phi)$, $\psi\in CX$ and $v>\Phi(\psi)$. Put
  $U=\{x\in X\mid\psi(x)>v\}$. By the discussion above,
  $U\cap\Zero(\Phi)=\varnothing$, hence $\psi(x)\le v$. Therefore we
  conclude that $x\in\Anti(\Phi)$.

  Consider now $\otimes=\luk$. Let $x\notin\Anti(\Phi)$. Then, there
  is some $\psi\in CX$ with $\psi(x)>\Phi(\psi)$. With $u=\psi(x)$, we
  obtain
  \[
    \hom(u,\psi(x))=1>\hom(u,\Phi(\psi))=
    u\pitchfork\Phi(\psi)\ge\Phi(u\pitchfork\psi),
  \]
  using Remark~\ref{d:rem:1} and that $\hom(u,-):[0,1]\to[0,1]$ is
  monotone and continuous. Therefore we may assume that
  $\psi(x)=1$. Since $\Phi(\psi)<1$, there is some $n\in\N$ with
  $\Phi(\psi)^n=0$, hence $\psi^n(x)=1$ and $\Phi(\psi^n)=0$. We
  conclude that $x\notin\Zero(\Phi)$.
\end{proof}

From the results above we obtain:

\begin{theorem}\label{thm:full-embedding-1}
  Assume that $\otimes=*$ or $\otimes=\luk$. Then the monad
  morphism $j$ between the monad $\mV$ on $\POSCH$ and the monad
  induced by the adjunction $C\dashv G$ of
  Proposition~\ref{prop:induced-dual-adjunction} is an
  isomorphism. Therefore the functor
  \[
    C:\POSCHDIST\longrightarrow\LaxMnd(\FinSup{[0,1]})^\op
  \]
  is fully faithful.
\end{theorem}

For $\Phi:CY\to CX$ in $\LaxMnd(\FinSup{[0,1]})$, the corresponding
distributor $\varphi:X\modto Y$ is given by
\[
  x\,\varphi\,y \iff y\in\bigcap_{\Phi(\psi)(x)=0}Z(\psi).
\]

From Corollary~\ref{cor:restriction-to-functions} one obtains:

\begin{corollary}
  Assume that $\otimes=*$ or $\otimes=\luk$. Then the functor
  \[
    C:\POSCH\longrightarrow\Mnd(\FinSup{[0,1]})^\op
  \]
  is fully faithful.
\end{corollary}

Theorem~\ref{thm:full-embedding-1} does not hold for an arbitrary
continuous quantale structure on $[0,1]$. For instance, for
$\otimes=\wedge$, every $\Phi=u\wedge-:[0,1]\to[0,1]$ ($u\in[0,1]$)
satisfies \Mon, \Act, \Sup{} (and also \LTen), but $V1$ contains only
two elements. Similarly, the map
\[
  \Phi:[0,1]\times[0,1]\longrightarrow[0,1],\;(u,v)\longmapsto
  u\vee\left(\frac{1}{2}\wedge v\right)
\]
has all these properties and also satisfies $\Phi(1)=1$; however,
$\Phi$ is not induced by a relation $1\relto 2$.

To deal with the general case, we introduce the following condition on
a map $\Phi:CX\to[0,1]$ where $\ominus$ denotes truncated minus on
$[0,1]$.

\begin{description}[font=\normalfont,labelindent=\parindent]
\item [\namedlabel{Min}{(Min)}] For every $u\in[0,1]$ and every
  $\psi\in CX$, $\Phi(\psi\ominus u)=\Phi(\psi)\ominus u$.
\end{description}

Clearly, for every closed upper subset $A\subseteq X$, the map
$\Phi_A:CX\to[0,1]$ satisfies \Min.

\begin{proposition}\label{prop:Ineq-with-Min}
  Let $X$ be a partially ordered compact space and $\Phi:CX\to[0,1]$ a
  map satisfying \Min. Then
  \[
    \Anti(\Phi)=\Zero(\Phi).
  \]
\end{proposition}
\begin{proof}
  Assume $x\notin\Anti(\Phi)$. Then there is some $\psi\in CX$ with
  $\psi(x)>\Phi(\psi)$. Put $u=\Phi(\psi)$. Then
  $\Phi(\psi\ominus u)=0$ and $(\psi\ominus u)(x)>0$, hence
  $x\notin\Zero(\Phi)$.
\end{proof}

Therefore we obtain:

\begin{proposition}
  Let $X$ be a partially ordered compact space. The map
  \[
    j_X:V X\longrightarrow \{\Phi:CX\to[0,1]\mid \text{$\Phi$ satisfies
      \Mon, \Act, \Sup, \LTen{} and \Min}\},\,A\longmapsto\Phi_A
  \]
  is bijective. If the quantale $[0,1]$ does not have nilpotent
  elements, then $j_X$ is bijective even if the condition \LTen{} is
  dropped on the right hand side.
\end{proposition}

Accordingly, we introduce the categories
\begin{align*}
  \Mnd_\ominus(\FinSup{[0,1]}) && \text{and} && \LaxMnd_\ominus(\FinSup{[0,1]})
\end{align*}
defined as $\Mnd(\FinSup{[0,1]})$ and $\LaxMnd(\FinSup{[0,1]})$
respectively, but the objects have an additional action
$\ominus:X\times[0,1]\to X$ and the morphisms preserve it. With the
action $\ominus:[0,1]\times[0,1]\to [0,1],\,(u,v)\mapsto u\ominus v$,
the $[0,1]$-category $[0,1]$ is an object of both categories.  As
before (see Proposition~\ref{prop:induced-dual-adjunction} and
Theorem~\ref{thm:full-embedding-1}), we obtain:

\begin{theorem}\label{thm:full-embedding-2}
  Under Assumption~\ref{ass:general-assumption}, the dualising object
  $([0,1]^\op,[0,1])$ induces a natural dual adjunction
  \[
    \POSCH\adjunctop{G}{C}\LaxMnd_\ominus(\FinSup{[0,1]})^\op.
  \]
  Here $CX$ is given by $\ORDCH(X,[0,1]^\op)$ with all operations
  defined pointwise, and $GA$ is the space
  $\LaxMnd_\ominus(\FinSup{[0,1]})(A,[0,1])$ equipped with the initial
  topology with respect to all evaluation maps
  \[
    \ev_a:\LaxMnd_\ominus(\FinSup{[0,1]})(A,[0,1])\longrightarrow[0,1]^\op,\,\Phi\longmapsto\Phi(a).
  \]
  Furthermore, we obtain a commutative diagram
  \[
    \xymatrix{\POSCHDIST\ar[rr]^{C} && \LaxMnd_\ominus(\FinSup{[0,1]})^\op,\\
      & \POSCH\ar[ul]\ar[ur]_C}
  \]
  of functors, and the induced monad morphism $j$ between $\mV$ and
  the monad induced by $C\dashv G$ is an isomorphism. Therefore the
  functor
  \[
    C:\POSCHDIST\longrightarrow\LaxMnd_\ominus(\FinSup{[0,1]})^\op
  \]
  is fully faithful, and so is the functor
  \[
    C:\POSCH\longrightarrow\Mnd_\ominus(\FinSup{[0,1]})^\op.
  \]
\end{theorem}

\begin{remark}
  Once we know that $C:\POSCH\to\Mnd_\ominus(\FinSup{[0,1]})^\op$ is
  fully faithful, we can add on the right hand side further operations
  if they can be transported pointwise from $[0,1]$ to $CX$. For
  instance, if $\hom(u,-):[0,1]\to[0,1]$ is continuous, then $CX$ has
  $u$-powers with $(\psi\pitchfork u)(x)=\hom(u,\psi(x))$, for all
  $x\in X$. Furthermore, every morphism $\Phi:CX\to CY$ in
  $\Mnd_\ominus(\FinSup{[0,1]})^\op$ preserves $u$-powers.
\end{remark}

In \citeyear{Ban83}, \citeauthor{Ban83} showed that $\COMPHAUS$ fully
embeds into the category of distributive lattices equipped with
constants from $[0,1]$ and constant preserving lattice
homomorphisms. As we pointed out in Remark~\ref{rem:Ban}, instead of
adding constants to the lattice $CX$ of continuous $[0,1]$-valued
functions, one could as well consider an action $u\wedge\psi$ of
$[0,1]$ on the lattice $CX$. Therefore \citeauthor{Ban83}'s result
should appear as a special case of Theorem~\ref{thm:full-embedding-1}
(for $\otimes=\wedge$). Unfortunately, this is not immediately the
case since we need the additional operation $\ominus$. However, using
some arguments of \citep{Ban83}, we finish this section showing that
$\Phi_{Z(\Phi)}=\Phi$, for every compact Hausdorff space and every
$\Phi:CX\to[0,1]$ in $\Mnd(\FinSup{[0,1]})$.

In analogy to Proposition~\ref{prop:functional-relation}, we have:
\begin{proposition}
  Let $X$ be a partially ordered compact space and assume that
  $\Phi:CX\to[0,1]$ satisfies \Mon, \Act, \Sup and \LTen.
  \begin{enumerate}
  \item If $\Phi$ satisfies also \Top, then
    $\Zero(\Phi)\neq\varnothing$.
  \item If $\Phi$ satisfies also \Ten, then $\Zero(\Phi)$ is
    irreducible (see Definition~\ref{d:def:1}).
  \end{enumerate}
\end{proposition}
\begin{proof}
  To see the first implication:
  $1=\Phi(1)\le\sup_{x\in \Zero(\Phi)}1$, hence
  $\Zero(\Phi)\neq\varnothing$. The proof of the second one is the
  same as the corresponding proof for
  Proposition~\ref{prop:functional-relation}.
\end{proof}

\begin{lemma}
  Let $X$ be a compact Hausdorff space and $\Phi:CX\to[0,1]$ in
  $\Mnd(\FinSup{[0,1]})$. We denote by $x_0$ the unique element of $X$
  with $\Zero(\Phi)=\{x_0\}$. Then, for every $\psi\in CX$,
  $\psi(x_0)=\Phi(\psi)$.
\end{lemma}
\begin{proof}
  By Proposition~\ref{prop:Phi-less-or-equal},
  $\Phi(\psi)\le\psi(x_0)$. To see the reverse inequality, let
  $u<\psi(x_0)$. Then $x_0\notin\{x\in X\mid \psi(x)\le u\}$,
  therefore there is some $\psi'\in CX$ with $\psi'(x_0)=0$ and
  $\psi'$ is constant $1$ on $\{x\in X\mid \psi(x)\le u\}$. Hence,
  $u\vee\psi'\le\psi\vee\psi'$. Since $\Phi(\psi')\le\psi'(x_0)=0$, we
  conclude that $u=\Phi(u)\le\Phi(\psi)$.
\end{proof}

\begin{theorem}
  Under Assumption~\ref{ass:general-assumption}, the functor
  \[
    C:\COMPHAUS\longrightarrow\Mnd(\FinSup{[0,1]})^\op
  \]
  is fully faithful.
\end{theorem}

\section{A Stone--Weierstra\ss{} theorem for $[0,1]$-categories}
\label{sec:stone-weierstrass-theorem}

In this section we adapt the classical Stone--Weierstra\ss{}
approximation theorem \citep{Sto48} to the context of
$[0,1]$-categories, which is an important step towards identifying the
image of the fully faithful functor
\[
  C:\POSCHDIST\longrightarrow\LaxMnd_\ominus(\FinSup{[0,1]})^\op.
\]
To do so, we continue working under
Assumption~\ref{ass:general-assumption}.

We recall that, for every partially ordered compact space $X$, the
$\V$-category $CX$ is finitely cocomplete with $[0,1]$-category
structure
\[
  d(\psi_1,\psi_2)=\inf_{x \in X}\hom(\psi_1(x), \psi_2(x)),
\]
for all $\psi_1,\psi_2 \in CX$. Furthermore, by
Theorem~\ref{thm:closure}, for all $M\subseteq CX$ and $\psi\in CX$ we
have $\psi\in\overline{M}$ if and only if, for every $u<1$, there is
some $\psi'\in M$ with $u\le d(\psi,\psi')$ and $u\le d(\psi',\psi)$.

\begin{lemma}
  \label{lem:stone}
  Let $L\subseteq CX$ be closed in $CX$ under finite suprema, the
  monoid structure and the action of $[0,1]$; that is, for all
  $\psi_1,\psi_2 \in L$ and $u \in [0,1]$, $\psi_1 \vee \psi_2 \in L$,
  $\psi_1 \otimes \psi_2 \in L$, $1\in L$ and
  $u \otimes \psi_1 \in L$. Let $\psi \in CX$. If the map
  $\hom:\im(\psi) \times [0,1] \to [0,1]$ is continuous and $L$
  satisfies the separation condition

  \begin{description}[font=\normalfont,labelindent=\parindent]
  \item [\namedlabel{Sep}{(Sep)}] for every $(x,y) \in X \times X$,
    with $x \ngeq y$, there exists $\psi \in L$ and an open
    neighbourhood $U_y$ of $y$ such that $\psi(x)=1$ and, for all
    $z \in U_y$, $\psi(z)=0$.
  \end{description}

  then $\psi \in \overline{L}$.
\end{lemma}
\begin{proof}
  Fix $x \in X$. Let $(\psi_y)_{y \in X}$ be the family of functions
  defined in the following way.
  \begin{itemize}
  \item If $y \nleq x$, then let $\psi_y$ be a function guaranteed
    by~\ref{Sep} and $U_y$ the corresponding neighborhood.
  \item If $y \leq x$, then $\psi_y$ is the constant function function
    $\psi(x)$.
  \end{itemize}

  By hypothesis, the functions $\hom(\psi(x),-):[0,1] \to [0,1]$ and
  $\psi$ are continuous. Thus, the set
  \[
    U_x=\{z \in X \mid u < \hom(\psi(x),\psi(z)) \}
  \]
  is an open neighborhood of every $y \leq x$, and for such $y \in X$
  we put $U_y=U_x$. Consenquently, the collection of sets $U_y$
  ($y \in X$) is an open cover of $X$. By compactness of $X$, there
  exists a finite subcover $U_{y_1}, \dots , U_{y_n}, U_x$ of
  $X$. Considering the corresponding functions
  $\psi_{y_1}, \dots , \psi_{y_n}, \psi_x$, we define
  $\phi_x=\psi_{y_1} \otimes \dots \otimes \psi_{y_n} \otimes \psi_x$.

  By construction, $\phi_x$ has the following properties:
  \begin{itemize}
  \item $\phi_x(x)=\psi(x)$, since $\psi_{y_i}(x)=1$ for
    $1 \leq i \leq n$ and $\psi_x(x)=\psi(x)$;
  \item for every $z \in X$, $u \otimes \phi_x(z) \leq \psi(z)$, since
    $z\in U_x$ or $z\in U_{y_i}$, for some $i$.
  \end{itemize}

  Now, for every $x \in X$ the set
  \[
    V_x=\{z \in X \mid u < \hom(\psi(z),\phi_x(z)) \}
  \]
  is open because the functions
  $\hom:\im(\psi) \times [0,1] \to [0,1]$, $\phi_x$ and $\psi$ are
  continuous. Therefore the collection of the sets $V_x$ is an open
  cover of $X$. Again, by compactness of $X$, there exists a finite
  subcover $V_{x_1}, \dots , V_{x_m}$ of $X$. By defining
  $\phi=\phi_{x_1} \vee \dots \vee \phi_{x_m}$ we obtain a function in
  $L$ such that for every $z \in X$:
  \begin{itemize}
  \item
    $\displaystyle{u \otimes \phi (z) = \bigvee_{j=1}^m u \otimes
      \phi_{x_j}(z) \leq \bigvee_{j=1}^m u \otimes \psi_{y_i}(z)}$,
    for some $\psi_{y_i}$ such that
    $u \otimes \psi_{y_i} \leq \psi(z)$;
  \item $u \otimes \psi(z) \leq \phi(z)$.\qedhere
  \end{itemize}
\end{proof}

\begin{remark}
  For the \L{}ukasiewicz tensor, the lemma above affirms that $L$ is
  dense in $CX$ in the usual sense since, in this case,
  \[
    \hom(u,v)\geq 1-\varepsilon \iff \max(v-u,0)\leq \varepsilon,
  \]
  for all $u,v\in [0,1]$. However, if the tensor product is
  multiplication, the function $\hom:[0,1]\times [0,1]\to [0,1]$ is
  not continuous in $(0,0)$; which requires us to add a further
  condition involving truncated minus (see
  Lemma~\ref{d:lem:1}). Finally, if the tensor is the infimum, we
  cannot expect to obtain a useful Stone--Weierstra\ss{} theorem using
  this closure. For example, for the partially ordered compact space
  $1=\{*\}$ the topology in $CX\simeq [0,1]$ is generated by the sets
  $\{u\}$ and $]u,1]$ with $u\neq 1$. For $x\neq 1$ and
  $M\subseteq [0,1]$, this means that the seemingly weaker condition
  $x\in \overline{M}$ implies that $x\in M$.
\end{remark}

\begin{lemma}\label{d:lem:2}
  Let $\otimes = \luk$ and $L \subseteq CX$. Assume that $L$ is closed
  in $CX$ under the monoid structure and $u$-powers, for all
  $u \in [0,1]$, and that the cone $(f:X\to [0,1]^\op)_{f \in L}$ is
  initial; that is, for all $x,y \in X$, $x \geq y$ if and only if,
  for all $\psi \in L$, $\psi(x) \leq \psi(y)$. Then $L$
  satisfies~\ref{Sep}.
\end{lemma}

\begin{proof}
  Let $(x,y) \in X \times X$ with $x \ngeq y$. By hyphotesis, there
  exists $\psi \in L$ and $c \in [0,1]$ such that
  $\psi(x) > c > \psi(y)$. Let $u=\psi(x)$. Since $L$ is closed for
  $u$-powers then $\psi'=\hom(u,\psi) \in L$. By
  corollary~\ref{cor:no-zero-div} there exists $n \in \N$ such that
  $c^n=0$. Therefore $\psi'^n(x)=1$ and for all
  $z \in U_y=\psi^{-1}[0,c[$, $\psi'^n(z)=0$.
\end{proof}

\begin{lemma}\label{d:lem:1}
  Let $\otimes = *$ and $L \subseteq CX$. Assume that $L$ is
  closed in $CX$ under $u$-powers and $- \ominus u$, for all
  $u \in [0,1]$, and that the cone $(f:X\to [0,1]^\op)_{f \in L}$ is
  initial. Then $L$ satisfies~\ref{Sep}.
\end{lemma}

\begin{proof}
  Let $(x,y) \in X \times X$ with $x \ngeq y$. By hyphotesis, there
  exists $\psi \in L$ and $c \in [0,1]$ such that
  $\psi(x) > c > \psi(y)$. Let $\psi'=\psi \ominus c$ and
  $u=\psi'(x)$. Let $\psi''=\hom(u, \psi') \in L$ and
  $U_y=\psi'^{-1}[0,c[$. Clearly, $\psi''(x)=1$ and, since $u > 0$,
  for all $z \in U_y$ we obtain $\psi''(z)=0$.
\end{proof}

The results above tell us that certain $[0,1]$-subcategories of $CX$
are actually equal to $CX$ if they are closed in $CX$. To ensure this
property, we will work now with Cauchy-complete
$[0,1]$-categories. First we have to make sure that the
$[0,1]$-category $CX$ is Cauchy-complete.

\begin{lemma}
  The subset
  \[
    \{(u,v)\mid u\le v\}\subseteq [0,1]\times [0,1]
  \]
  of the $[0,1]$-category $[0,1]\times[0,1]$ is closed.
\end{lemma}
\begin{proof}
  Just observe that $\{(u,v)\mid u\le v\}$ can be presented as the
  equaliser of the $[0,1]$-functors $\wedge:[0,1]\times[0,1]\to[0,1]$
  and $\pi_1:[0,1]\times[0,1]\to[0,1]$.
\end{proof}

\begin{corollary}
  For every ordered set $X$ (with underlying set $|X|$), the subset
  \[
    \ORD(X,[0,1]^\op)\subseteq [0,1]^{|X|}
  \]
  of the $[0,1]$-category $[0,1]^{|X|}$ is closed.
\end{corollary}

Recall that we write $\xi:U[0,1]\to[0,1]$ for the convergence of the
Euclidean topology of $[0,1]$.
\begin{lemma}
  For every set $X$ and every ultrafilter $\fx$ on $X$, the map
  \[
    \Phi_\fx:[0,1]^X\to[0,1],\,\psi\mapsto\xi\cdot U\psi(\fx)
  \]
  is a $[0,1]$-functor.
\end{lemma}
\begin{proof}
  Since domain and codomain of $\Phi_\fx$ are both $\V$-copowered, the
  assertion follows from
  \begin{align*}
    \xi\cdot U\psi\le\xi\cdot U\psi' &&\text{and}&& \xi\cdot U(\psi\otimes u)=(\xi\cdot U\psi)\otimes u,
  \end{align*}
  for all $u\in[0,1]$ and $\psi,\psi'\in [0,1]^X$ with $\psi\le\psi'$.
\end{proof}

\begin{corollary}
  For every compact Hausdorff space $X$, the subset
  \[
    \COMPHAUS(X,[0,1])\subseteq [0,1]^{|X|}
  \]
  of the $[0,1]$-category $[0,1]^{|X|}$ is closed.
\end{corollary}
\begin{proof}
  For an ultrafilter $\fx\in UX$ with convergence point $x\in X$, a
  map $\psi:X\to[0,1]$ preserves this convergence if and only if
  $\psi$ belongs to the equaliser of $\Phi_\fx$ and $\pi_x$.
\end{proof}

\begin{proposition}
  For every partially ordered compact space $X$, the $[0,1]$-category
  $CX$ is Cauchy-complete.
\end{proposition}

We will now introduce a category $\catA$ of $[0,1]$-categories which
depends on the chosen tensor $\otimes$ on $[0,1]$.

\begin{description}
\item[For the \L{}ukasiewicz tensor $\otimes=\luk$] $\catA$ is the
  category with objects all $[0,1]$-powered objects in the category
  $\Mnd(\FinSup{[0,1]})$, and morphisms all those arrows in
  $\Mnd(\FinSup{[0,1]})$ which preserve powers by elements of $[0,1]$.
\item[For the multiplication $\otimes=*$] $\catA$ is the category with
  objects all $[0,1]$-powered objects in the category
  $\Mnd_\ominus(\FinSup{[0,1]})$, and morphisms all those arrows in
  $\Mnd_\ominus(\FinSup{[0,1]})$ which preserve powers by elements of
  $[0,1]$.
\end{description}

\begin{remark}
  The category $\catA$ over $\SET$ is $\aleph_1$-ary quasivariety and,
  moreover, a full subcategory of a finitary variety. Therefore the
  isomorphisms in $\catA$ are precisely the bijective morphisms.
\end{remark}

\begin{theorem}\label{thm.stone-weierstrass}
  Assume that $\otimes=*$ is multiplication or $\otimes=\luk$ is
  the \L{}ukasiewicz tensor. Let $m:A\to CX$ be an injective morphism
  in $\catA$ so that the cone $(m(a):X\to[0,1]^\op)_{a\in A}$ is
  point-separating and initial with respect to the canonical forgetful
  functor $\POSCH\to\SET$. Then $m$ is an isomorphism in $\catA$ if
  and only if $A$ is Cauchy-complete.
\end{theorem}
\begin{proof}
  Clearly, if $m$ is an isomorphism, then $A$ is Cauchy-complete since
  $CX$ is so. The reverse implication is clear for $\otimes=\luk$ by
  Lemmas~\ref{lem:stone} and \ref{d:lem:2}. Consider now $\otimes=*$
  multiplication. Let $\psi\in CX$. Put
  $\psi'=\frac{1}{2}*\psi+\frac{1}{2}$, then $\psi'$ is monotone and
  continuous. By Lemmas~\ref{lem:stone} and \ref{d:lem:1},
  $\psi'\in\im(m)$ and therefore also
  $\psi=\hom(\frac{1}{2},\psi'\ominus\frac{1}{2})$ belongs to $CX$.
\end{proof}

We say that an object $A$ of $\catA$ \df{has enough characters}
whenever the cone $(\varphi:A\to[0,1])_\varphi$ of all morphisms into
$[0,1]$ separates the points of $A$.

\begin{corollary}
  Let $A$ be an object in $\catA$. Then $A\simeq CX$ in $\catA$ for
  some partially ordered compact space $X$ if and only if $A$ is
  Cauchy-complete and has enough characters.
\end{corollary}
\begin{proof}
  If $A\simeq CX$ in $\catA$, then clearly $A$ is Cauchy-complete and
  has enough characters. Assume now that $A$ has this
  properties. Then, by \citep[Proposition 2.4]{LR79},
  $X=\hom(A,[0,1])$ is a partially ordered compact space with the
  initial structure relative to all evaluation maps
  $\ev_a:X\to[0,1]^\op$ ($a\in A$). The map
  $m:A\to CX,\,a\mapsto ev_a$ is injective since $A$ has enough
  characters and satisfies the hypothesis of
  Theorem~\ref{thm.stone-weierstrass}, hence $m$ is an isomorphism.
\end{proof}

Therefore we consider now the following categories.
\begin{itemize}
\item $\catA_{[0,1],\cc}$ denotes the full subcategory of $\catA$
  defined by the Cauchy-complete objects having enough characters.
\item $\catB_{[0,1],\cc}$ denotes the category with the same objects
  as $\catA_{[0,1],\cc}$, and the morphisms of $\catB_{[0,1],\cc}$ are
  the finitely cocontinuous $[0,1]$-functors which laxly preserve the
  multiplication.
\end{itemize}

\begin{theorem}
  For $\otimes=*$ multiplication or $\otimes=\luk$ the \L{}ukasiewicz
  tensor,
  \begin{align*}
    \POSCHDIST\simeq\catB_{[0,1],\cc}^\op &&\text{and}&& \POSCH\simeq\catA_{[0,1],\cc}^\op.
  \end{align*}
\end{theorem}

\section{Metric compact Hausdorff spaces and metric Vietoris monads}
\label{sec:metr-comp-hausd}

As we pointed already out in Remark~\ref{rem:Kleisli-adjunction}, the
constructions leading to dualities for Kleisli categories seem to be
more ``canonical'' if we work with a monad $\mT=\monad$ satisfying
$T1\simeq[0,1]^\op$. In \citep{Hof14} we introduce a generalisation of
the Vietoris monad with this property in the context of ``enriched
topological spaces'', more precisely, models of topological theories
as defined in \citep{Hof07}. Such a topological theory involves a
$\SET$-monad and a quantale together with an Eilenberg--Moore algebra
structure on the underlying set of the quantale, subject to further
axioms. In this paper we consider only the ultrafilter monad
$\mU=\umonad$ and a quantale with underlying lattice $[0,1]$. The
convergence of the Euclidean compact Hausdorff topology on $[0,1]$
defines an Eilenberg--Moore algebra structure for the ultrafilter
monad:
\[
  \xi:U[0,1]\to[0,1],\, \xi(\fx)=\sup_{A\in\fx}\inf_{u\in
    A}u=\inf_{A\in\fx}\sup_{u\in A}u.
\]
We continue assuming that the multiplication of the quantale $[0,1]$
is continuous and has $1$ as neutral element; that is, we continue
working under Assumption~\ref{ass:general-assumption}.  Under these
conditions, $\thU=\utheory$ is a strict topological theory as defined
in \citep{Hof07}. In order to keep the amount of background theory
small, we do not enter here into the details of monad-quantale
enriched categories but give only the details needed to understand the
Kleisli category of the $[0,1]$-enriched Vietoris monad. We refer to
\citep[Section 1]{Hof14} for an overview, and a comprehensive
presentation of this theory can be found in \citep{HST14}.

The functor $U:\SET\to\SET$ extends to a 2-functor
$U_\xi:\Rels{[0,1]}\to\Rels{[0,1]}$ where $U_\xi X=UX$ for every set
$X$ and
\[
  U_\xi r(\fx,\fy)=\{\xi\cdot Ur(\fw)\mid\fw\in
    U(X\times Y),
    U\pi_1(\fw)=\fx,U\pi_2(\fw)=\fy\}=\sup_{A\in\fx,B\in\fy}\inf_{x\in
    A,y\in B}r(x,y)
\]
for all $r:X\relto Y$ in $\Rels{[0,1]}$, $\fx\in UX$ and $\fy\in UY$.

\begin{definition}
  A \df{$\thU$-category} is a set $X$ equipped with a $[0,1]$-relation
  \[
    a:UX\times X\longrightarrow [0,1],
  \]
  subject to the axioms
  \begin{align*}
    1=a(e_X(x),x) &&\text{and}&& U_\xi a(\fX,\fx)\otimes a(\fx,x)\le a(m_X(\fX),x),
  \end{align*}
  for all $\fX\in UUX$, $\fx\in UX$ and $x\in X$. Given
  $\thU$-categories $(X,a)$ and $(Y,b)$, a \df{$\thU$-functor}
  $f:(X,a)\to(Y,b)$ is a map $f:X\to Y$ such that
  \[
    a(\fx,x)\le b(Uf(\fx),f(x)),
  \]
  for all $\fx\in UX$ and $x\in X$.
\end{definition}

Clearly, the composite of $\thU$-functors is a $\thU$-functor, and so
is the identity map $1_X:(X,a)\to(X,a)$, for every $\thU$-category
$(X,a)$. We denote the category of $\thU$-categories and
$\thU$-functors by $\Cats{\thU}$. The most notable example is
certainly the case of $\otimes=*$ being multiplication: since
$[0,1]\simeq[0,\infty]$, $\Cats{\thU}$ is isomorphic to the category
$\APP$ of approach spaces and non-expansive maps (see \citep{Low97}).

The category $\Cats{\thU}$ comes with a canonical forgetful functor
$\Cats{\thU}\to\SET$ which is topological, hence $\Cats{\thU}$ is
complete and cocomplete and the forgetful functor $\Cats{\thU}\to\SET$
preserves limits and colimits. For example, the product of
$\thU$-categories $(X,a)$ and $(Y,b)$ can be constructed by first
taking the Cartesian product $X\times Y$ of the sets $X$ and $Y$, and
then equipping $X\times Y$ with the structure $c$ defined by
\[
  c(\fw,(x,y))=a(U\pi_1(\fw),x)\wedge b(U\pi_2(\fw),y),
\]
for all $\fw\in U(X\times Y)$, $x\in X$ and $y\in Y$. More important
to us is, however, a different structure $c$ on $X\times Y$ derived
from the tensor product of $[0,1]$, namely
\[
  c(\fw,(x,y))=a(U\pi_1(\fw),x)\otimes b(U\pi_2(\fw),y).
\]
We denote this $\thU$-category as $(X,a)\otimes(Y,b)$; in fact, this
construction extends naturally to morphisms and yields a functor
$-\otimes-:\Cats{\thU}\times\Cats{\thU}\to\Cats{\thU}$.

Every $\thU$-category $(X,a)$ defines a topology on the set $X$ with
convergence
\[
  \fx\to x \iff 1=a(\fx,x),
\]
and this construction defines a functor $\Cats{\thU}\to\TOP$ which
commutes with the forgetful functors to $\SET$. The functor
$\Cats{\thU}\to\TOP$ has a left adjoint $\TOP\to\Cats{\thU}$ which
sends a topological space to the $\thU$-category with the same
underlying set, say $X$, and with the discrete convergence
\[
  UX\times X\longrightarrow[0,1],\,(\fx,x)\longmapsto
  \begin{cases}
    1 & \text{if $\fx\to x$},\\
    0 & \text{otherwise.}
  \end{cases}
\]
This functor allows us to interpret topological spaces as
$\thU$-categories. Note that, for a topological space $X$ and a
$\thU$-category $Y$, we have $X\otimes Y=X\times Y$ in
$\Cats{\thU}$. There is also a faithful functor
$\Cats{\thU}\to\Cats{[0,1]}$ which commutes with the forgetful
functors to $\SET$ and sends a $\thU$-category $(X,a)$ to the
$[0,1]$-category $(X,a_0)$ where $a_0=a\cdot e_X$; hence
$a_0(x,y)=a(e_X(x),y)$, for all $x,y\in X$. We also note that the
natural order of the underlying topology of an $\thU$-category $(X,a)$
coincides with the order induced by the $[0,1]$-category $(X,a_0)$. We
refer to this order as the underlying order of $(X,a)$.

\begin{remark}
  An important example of a $\thU$-category is given by $[0,1]$ with
  convergence $(\fx,u)\mapsto\hom(\xi(\fx),u)$. In the underlying
  topology of $[0,1]$, $\fx\to u$ precisely when $\xi(\fx)\le u$;
  hence, the closed subsets are precisely the intervals $[v,1]$. From
  now on $[0,1]$ refers to this $\thU$-category; to distinguish,
  $[0,1]_e$ denotes the standard compact Hausdorff space with
  convergence $\xi$, and $[0,1]_o$ denotes the ordered compact
  Hausdorff space with the usual order and the Euclidean topology.
\end{remark}

The enriched Vietoris monad $\mV=\vmonad$ on $\Cats{\thU}$ sends a
$\thU$-category $X$ to the $\thU$-category $VX$ with underlying set
\[
  \{\varphi:X\to[0,1]\mid \varphi\text{ is a $\thU$-functor}\},
\]
the underlying $[0,1]$-category structure of the $\thU$-category $VX$
is given by
\[
  [\varphi,\varphi']=\inf_{x\in X}\hom(\varphi'(x),\varphi(x)),
\]
and therefore
\[
  \varphi\le \varphi'\iff \varphi(x)\ge \varphi'(x),\text{ for all
  }x\in X
\]
in its underlying order. It is shown in \citep{Hof14} that this monad
restricts to the $[0,1]$-enriched counterpart of the category of
stably compact spaces and spectral maps: the category of separated
representable $\thU$-categories. A $\thU$-category $(X,a)$ is
\df{representable} whenever $a\cdot U_\xi a=a\cdot m_X$ and there is a
map $\alpha:UX\to X$ with $a=a_0\cdot\alpha$. If $(X,a)$ is separated,
then $\alpha:UX\to X$ is unique and an $\mU$-algebra structure on $X$;
that is, the convergence of a compact Hausdoff topology on $X$. The
separated representable $\thU$-categories are the objects of the
category $\Reps{\thU}$, a morphism $f:(X,a)\to(Y,b)$ in $\Reps{\thU}$
is a $[0,1]$-functor $f:(X,a_0)\to(Y,b_0)$ where $f$ is also
continuous with respect to the corresponding compact Hausdoff
topologies; that is, $f\cdot\alpha=\beta\cdot Uf$. We also note that
the category $\Reps{\thU}$ is complete and the inclusion functors
$\Reps{\thU}\to\Cats{\thU}$ preserves limits. For $(X,a)$ in
$\Reps{\thU}$ with $a=a_0\cdot\alpha$, also $(X,a_0^\circ\cdot\alpha)$
is a separated representable $\thU$-category, called the \df{dual} of
$(X,a)$ and denoted as $(X,a)^\op$. In fact, this construction defines
a functor $(-)^\op:\Reps{\thU}\to\Reps{\thU}$ leaving maps
unchanged. The $\thU$-category $[0,1]$ is separated and representable,
with $\hom:\V\relto\V$ being the underlying $[0,1]$-category structure
and $\xi:U[0,1]\to[0,1]$ the convergence of the corresponding compact
Hausdoff topology. We note that now $V1$ is isomorphic to
$[0,1]^\op$. Below we collect some important properties.

\begin{proposition}\label{prop:1}
  The following assertions hold.
  \begin{enumerate}
  \item The maps
    \begin{align*}
      \wedge:[0,1]\times[0,1]&\to[0,1], & \vee:[0,1]\times[0,1]&\to[0,1],  & \otimes:[0,1]\otimes[0,1]&\to[0,1]
    \end{align*}
    are morphisms in $\Reps{\thU}$.
  \item For every $u\in[0,1]$, $u\otimes-:[0,1]\to[0,1]$ is a morphism
    of $\Reps{\thU}$. Therefore the map $u\otimes-$ preserves
    non-empty infima.
  \item Let $u\in[0,1]$ so that $\hom(u,-)$ is continuous of type
    $\hom(u,-):[0,1]_e\to[0,1]_e$. Then $\hom(u,-):[0,1]\to[0,1]$ is a
    morphism in $\Reps{\thU}$.
  \item $\inf:[0,1]^I\to[0,1]$ is a $\thU$-functor, for every set $I$.
  \item For every $v\in[0,1]$, $\hom(-,v):[0,1]^\op\to[0,1]$ is a
    $\thU$-functor.
  \end{enumerate}
\end{proposition}
\begin{proof}
  The first assertion follows directly from our
  Assumption~\ref{ass:general-assumption}. Note that
  $[0,1]\otimes[0,1]$ is in $\Reps{\thU}$, see \citep[Remark
  4.9]{Hof13}. The second assertion is a direct consequence of the
  first one, and the third one follows form $\hom(u,-):[0,1]\to[0,1]$
  being a $[0,1]$-functor. Regarding the forth assertion, see
  \citep[Corollary 5.3]{Hof07}. Regarding the last assertion, the
  condition of \citep[Lemma 5.1]{Hof07} can be verified using
  \citep[Lemma 3.2]{Hof07}.
\end{proof}

\begin{remark}
  Similarly to the connection between stably compact spaces and
  partially ordered compact spaces, representable $\thU$-categories
  can be seen as compact Hausdorff spaces with a compatible
  $[0,1]$-category structure. More precisely, in \citep{Tho09} it is
  shown that the $\SET$-monad $\mU$ extends to a monad on
  $\Cats{[0,1]}$, and there is a natural comparison functor
  $K:(\Cats{[0,1]})^\mU\to\Cats{\thU}$ sending a $[0,1]$-category
  $(X,a_0)$ with Eilenberg--Moore algebra structure
  $\alpha:U(X,a_0)=(UX,U_\xi a_0)\to(X,a_0)$ to the $\thU$-category
  $(X,a_0\cdot\alpha)$. The functor $K$ restricts to an equivalence
  between the full subcategory of $(\Cats{[0,1]})^\mU$ defined by all
  separated $[0,1]$-categories and the category $\Reps{\thU}$ (see
  also \citep{Hof14} for details).
\end{remark}

We do not need to say much about the enriched Vietoris monad
$\mV=\vmonad$ here, it is enough to have a better understanding of the
Kleisli category $\Reps{\thU}_\mV$. A morphism $X\to VY$ in
$\Reps{\thU}$ corresponds to a $[0,1]$-distributor between the
underlying $[0,1]$-categories, we call such a distributor a
\df{continuous $[0,1]$-distributor} between the separated
representable $\thU$-categories $X$ and $Y$. Similar to the classical
case, composition in $\Reps{\thU}_\mV$ corresponds to composition of
$[0,1]$-relations, and the identity morphism on $X$ is given by $a_0$
(see \citep[Section 8]{Hof14}). Therefore we identify
$\Reps{\thU}_{\mV}$ with the category $\RepDists{\thU}$ of separated
representable $\thU$-categories and continuous $[0,1]$-distributors
between them, and with the compositional structure and identities as
described above.

The adjunction
\[
  \Cats{\thU}\adjunct{\text{discrete}}{\text{forgetful}}\TOP
\]
restricts to an adjunction
\[
  \Reps{\thU}\adjunct{\text{discrete}}{\text{forgetful}}\STCOMP\simeq\POSCH,
\]
which allows us to transfer $\mV$ to a monad $\mV=\vmonad$ on
$\POSCH$. The Kleisli category $\POSCH_\mV$ for this monad can be
identified with the full subcategory $\POSCHDists{\thU}$ of
$\RepDists{\thU}$ defined by the $\thU$-categories in the image of
$\POSCH\xrightarrow{\text{discrete}}\Reps{\thU}$. Furthermore, the
category $\POSCHDIST$ can be identified with the subcategory of
$\POSCHDists{\thU}$ defined by those distributors $\varphi:X\modto Y$
where the map $\varphi:X\times Y\to[0,1]$ takes only values in
$\{0,1\}$. The functor
$\Reps{\thU}\xrightarrow{\text{forgetful}}\POSCH$ sends the
$\thU$-category $[0,1]$ to the ordered compact Hausdorff space
$[0,1]_o$ and $[0,1]^\op$ to $[0,1]_o^\op$, the latter follows from
the fact that this functor commutes with dualisation. Finally, there
is also a canonical adjunction
\[
  \Reps{\thU}\adjunct{\text{discrete}}{\text{forgetful}}\COMPHAUS
\]
sending a separated representable $\thU$-category to its corresponding
compact Hausdorff space (see also \citep[Remark 2.6]{Hof14}).

Since the construction of dual adjunctions typically involves initial
lifts, below we give a description of initial cones in $\Reps{\thU}$.
\begin{proposition}\label{d:prop:1}
  Let $(\psi_i:(X,a)\to (X_i,a_i))_{i\in I}$ be a point-separating
  cone in $\Reps{\thU}$. Then the following assertions are equivalent.
  \begin{tfae}
  \item For all $x,y\in X$,
    $\displaystyle{a_0(x,y)=\inf_{i\in
        I}{a_i}_0(\psi_i(x),\psi_i(y))}$.
  \item The cone $(\psi:X\to X_i)_{i\in I}$ is initial with respect to
    the forgetful functor $\Reps{\thU}\to\COMPHAUS$.
  \item The cone $(\psi:X\to X_i)_{i\in I}$ is initial with respect to
    the forgetful functor $\Reps{\thU}\to\SET$.
  \item The cone $(\psi:X\to X_i)_{i\in I}$ is initial with respect to
    the forgetful functor $\Cats{\thU}\to\SET$.
  \end{tfae}
\end{proposition}
\begin{proof}
  This follows from the description of initial structures for the
  functor $(\Cats{[0,1]})^\mU\to\Cats{[0,1]}$ given in
  \citep[][Proposition 3]{Tho09}, the fact that every point-separating
  cone is initial with respect to $\COMPHAUS\to\SET$, and the
  description of initial structures for the functor
  $\Cats{\thU}\to\SET$ in \citep[][Proposition III.3.1.1]{HST14}.
\end{proof}

\begin{definition}
  A point-separating cone in $\Reps{\thU}$ is called \df{initial}
  whenever it satisfies the first and hence all of the conditions
  of Proposition~\ref{d:prop:1}.
\end{definition}

\section{Duality theory for continuous enriched
  distributors}\label{sec:enriched-Vietoris}

In this section we will use the setting described in
Section~\ref{sec:metr-comp-hausd} and aim for results similar to the
ones obtained in Section~\ref{sec:duality-distributors} for ordered
compact Hausdorff spaces. To do so, we continue working under
Assumption~\ref{ass:general-assumption}. By Propositions~\ref{prop:1}
and \ref{prop:init}, the dualising object $([0,1]^\op,[0,1])$ induces
a natural dual adjunction
\[
  \Reps{\thU}\adjunctop{G}{C}\FinSup{[0,1]}^\op;
\]
here $CX$ has as underlying set all morphisms $X\to[0,1]^\op$ in
$\Reps{\thU}$. For every separated representable $\thU$-category $X$,
the map
\[
  \hom(X,[0,1]^\op)\longrightarrow\hom(VX,[0,1]^\op),\,\psi
  \longmapsto(\varphi \mapsto\psi\cdot\varphi =\sup_{x\in
    X}\psi(x)\otimes\varphi(x))
\]
is certainly a morphism $CX\to CVX$ in $\FinSup{[0,1]}$, by
Theorem~\ref{thm:Kleisli-adjunction} and
Remark~\ref{rem:Kleisli-adjunction} we obtain a commutative diagram
\[
  \xymatrix{\RepDists{\thU}\ar[rr]^-{C} && \FinSup{[0,1]}^\op\\
    & \Reps{\thU}\ar[ul]\ar[ur]_C}
\]
of functors. For $\varphi:X\modto Y$ in $\RepDists{\thU}$, we have
\[
  C\varphi:CY\longrightarrow CX,\;\psi\longmapsto\psi\cdot\varphi.
\]
If $\varphi$ lives in $\POSCHDIST$, then $C\varphi$ coincides with
what was defined in the Section~\ref{sec:duality-distributors}.

\begin{remark}\label{rem:tensor-in-metric-case}
  If $X$ is a partially ordered compact space, then, for all
  $\psi_1,\psi_2\in CX$, the composite
  \[
    X\xrightarrow{\Delta_X}X\times X\simeq X\otimes
    X\xrightarrow{\psi_1\otimes\psi_2}[0,1]^\op\otimes[0,1]^\op\simeq([0,1]\otimes[0,1])^\op\xrightarrow{\otimes^\op}[0,1]^\op
  \]
  is also in $\Reps{\thU}$. Therefore we can still consider
  $\psi_1\otimes\psi_2\in CX$; however, $C\varphi$ does not need to
  preserve this operation, not even laxly.
\end{remark}

The functor $C:\RepDists{\thU}\to\FinSup{[0,1]}^\op$ induces a monad
morphism $j$ whose component at $X$ is given by the maps
\begin{equation}\label{formula.metric_monad_morphism}
  j_X:VX\longrightarrow\FinSup{[0,1]}(CX,[0,1]),\;(\varphi:1\modto X)\longmapsto\left(\psi\mapsto\psi\cdot\varphi=\sup_{x\in X}(\psi(x)\otimes\varphi(x))\right).
\end{equation}

For every $\Phi:CX\to[0,1]$ in $\FinSup{[0,1]}$, we define a map
$\varphi:X\to[0,1]$ by
\[
  \varphi(x)=\inf_{\psi\in CX}\hom(\psi(x),\Phi(\psi)).
\]
For every $\psi\in CX$,
\[
  X\xrightarrow{\,\psi\,}[0,1]^\op\xrightarrow{\,\hom(-,\Phi(\psi))\,}[0,1]
\]
is a $\thU$-functor, and so is $\varphi:X\to[0,1]$ since it can be
written as the composite (with $I=CX$)
\[
  X\longrightarrow[0,1]^I\xrightarrow{\,\inf\,}[0,1].
\]
In other words, $\varphi\in VX$. For $\Phi,\Phi':CX\to[0,1]$ in
$\FinSup{[0,1]}$ with corresponding map $\varphi,\varphi':X\to[0,1]$, 
\begin{align*}
  [\varphi',\varphi]
  &=\inf_{x\in X}\hom(\inf_{\overline{\psi}\in CX}\hom(\overline{\psi}(x),\Phi(\overline{\psi})),\inf_{\psi\in CX}\hom(\psi(x),\Phi'(\psi)))\\
  &=\inf_{x\in X}\inf_{\psi\in CX}\hom(\inf_{\overline{\psi}\in CX}\hom(\overline{\psi}(x),\Phi(\overline{\psi})),\hom(\psi(x),\Phi'(\psi)))\\
  &\ge\inf_{x\in X}\inf_{\psi\in CX}\hom(\hom(\psi(x),\Phi(\psi)),\hom(\psi(x),\Phi'(\psi)))\\
  &\ge\inf_{x\in X}\inf_{\psi\in CX}\hom(\Phi(\psi),\Phi'(\psi)) && (\text{$\hom(\psi(x),-)$ is a $[0,1]$-functor})\\
  &=[\Phi',\Phi].
\end{align*}
Therefore $\Phi\mapsto\varphi$ defines a $[0,1]$-functor $GCX\to
VX$. Furthermore, one easily verifies that these constructions define
an adjunction $\Cats{[0,1]}$:

\begin{proposition}\label{prop:j-adjunction}
  Let $X$ be a separated representable $\thU$-category. Then the
  following assertions hold.
  \begin{enumerate}
  \item\label{item:7} For every $\varphi\in VX$ and every $x\in X$,
    \[
      \varphi(x)\le\inf_{\psi\in CX}\hom(\psi(x),\sup_{y\in
        X}\psi(y)\otimes\varphi(y)).
    \]
  \item For every $\Phi:CX\to[0,1]$ in $\FinSup{[0,1]}$ and every
    $\psi\in CX$,
    \[
      \sup_{x\in X}\psi(x)\otimes\inf_{\psi'\in
        CX}\hom(\psi'(x),\Phi(\psi'))\le\Phi(\psi).
    \]
  \end{enumerate}
\end{proposition}

Recall from Proposition~\ref{prop:initial-cogenerator} that
$[0,1]_o^\op$ is an initial cogenerator in $\POSCH$; so far we do not
know if $[0,1]^\op$ is an initial cogenerator in
$\Reps{\thU}$. Therefore we define:

\begin{definition}
  A separated representable $\thU$-category $X$ is called
  \df{$[0,1]^\op$-cogenerated} if the cone
  $(\psi:X\to[0,1]^\op)_{\psi\in CX}$ is point-separating and initial.
\end{definition}
Of course, $[0,1]^\op$ is $[0,1]^\op$-cogenerated, and so is every
partially ordered compact space. Our next question is whether
$V:\Reps{\thU}\to\Reps{\thU}$ restricts to $[0,1]^\op$-cogenerated
representable $\thU$-categories.  

\begin{lemma}\label{lem:1}
  Assume that the tensor product $\otimes$ on $[0,1]$ is either
  $*$, $\luk$ or $\wedge$. Let $X$ be
  $[0,1]^\op$-cogenerated. Then, for all $x,y\in X$,
  \[
    a_0(y,x)=\inf_{\psi\in CX,\psi(x)=1}\psi(y).
  \]
\end{lemma}
\begin{proof}
  Since $X$ is $[0,1]^\op$-cogenerated, we certainly have
  \[
    a_0(y,x)=\inf_{\psi\in CX}\hom(\psi(x),\psi(y))\le\inf_{\psi\in
      CX,\psi(x)=1}\psi(y).
  \]
  Assume first that $\otimes=*$ or $\otimes=\luk$. For $\psi\in
  CX$, put $u=\psi(x)$. Then
  \[
    \hom(\psi(x),\psi(y))=(u\pitchfork\psi)(y)
  \]
  and $(u\pitchfork\psi)(x)=\hom(u,\psi(x))=1$. Since
  $u\pitchfork\psi\in CX$, the assertion follows. 

  Assume now that $\otimes=\wedge$. The assertion follows immediately
  if $a_0(y,x)=1$. Let now $\psi\in CX$, we may assume that
  $\psi(x)>\psi(y)$. Let $b\in[0,1]$ with $\psi(y)\le
  b<\psi(x)$. Consider the piecewise linear map $h:[0,1]\to[0,1]$ with
  $h(v)=v$ for all $v\le b$ and $h(u)=1$ for all $u\ge\psi(x)$. Then
  $h^\op\cdot\psi\in CX$ since $h:[0,1]_e\to[0,1]_e$ is continuous and
  $h:[0,1]\to[0,1]$ is a $[0,1]$-functor. To see the latter, let
  $u,v\in[0,1]$.  If $\hom(u,v)=1$, then $\hom(h(u),h(v))=1$ since $h$
  is monotone. Assume now $u>v$. We distinguish the following cases.
  \begin{description}
  \item[If $\psi(x)\le v$] In this case,
    $\hom(h(u),h(v))=1\ge \hom(u,v)$.
  \item[If $v<\psi(x)\le u$] Now we have
    $\hom(u,v)=v\le h(v)=\hom(h(u),h(v)$.
  \item[If $v<u<\psi(x)$] Assume first that $b\le u$. Then
    $\hom(u,v)=v$ and $v\le\hom(h(u),h(v))$ since
    $v\otimes h(u)\le v\le h(v)$. Finally, if $v<u<b$, then the
    assertion follows from $u=h(u)$ and $v=h(v)$.
  \end{description}
  We conclude that $h(\psi(y))=\psi(y)$ and $h(\psi(x))=1$.
\end{proof}

Also note that, for every $\psi:X\to[0,1]^\op$ in $\Reps{\thU}$, the
canonical extension $\psi^\Diamond:VX\to[0,1]^\op$ of $\psi$ to the
free $\mV$-algebra $VX$ over $X$ sends $\varphi\in VX$ to
$\sup_{x\in X}\varphi(x)\otimes\psi(x)$; and the diagram
\[
  \xymatrix{VX\ar[dr]_{\psi^\Diamond}\ar[r]^-{j_X} &
    GC(X)\ar[d]^{\ev_\psi}\\ & [0,1]^\op}
\]
commutes.

\begin{lemma}\label{lem:psi-box-pointseparating}
  Assume that the tensor product $\otimes$ on $[0,1]$ is either
  $*$, $\luk$ or $\wedge$. For every $[0,1]^\op$-cogenerated $X$
  in $\Reps{\thU}$, the cone
  $(\psi^\Diamond:VX\to[0,1]^\op)_{\psi\in CX}$ is point-separating.
\end{lemma}
\begin{proof}
  Let $\varphi_1,\varphi_2\in VX$ and $x\in X$ with
  $\varphi_1(x)<u<\varphi_2(x)$. Let $y\in X$. By Lemma~\ref{lem:1}, and
  since every $v\otimes-:[0,1]\to[0,1]$ preserves non-empty infima,
  \[
    \inf_{\psi\in
      CX,\psi(x)=1}(\varphi_1(y)\otimes\psi(y))=\varphi_1(y)\otimes\left(\inf_{\psi\in
        CX,\psi(x)=1}\psi(y)\right)=\varphi_1(y)\otimes
    a_0(y,x)\le\varphi_1(x)<u.
  \]
  Therefore there is some $\psi_y\in CX$ with $\psi_y(x)=1$ and
  $\varphi_1(y)\otimes\psi_y(y)<u$. The composite
  \[
    (X,\alpha)\xrightarrow{\langle\varphi_1,\psi_y\rangle}[0,1]\times[0,1]_e\simeq[0,1]\otimes[0,1]_e\longrightarrow
    [0,1]\otimes[0,1]\xrightarrow{\otimes}[0,1]
  \]
  is in $\Cats{\thU}$ and therefore also continuous with respect to
  the underlying topologies, which tells us that the set
  \[
    V_y=\{z\in X\mid \varphi_1(z)\otimes\psi_y(z)<u\}
  \]
  is open in the compact Hausdoff space $(X,\alpha)$.

  By construction, $(V_y)_{y\in X}$ is an open cover of the compact
  Hausdorff space $(X,\alpha)$; therefore we find $n\in\N$ and
  $y_1,\dots,y_n$ in $X$ with $X=V_{y_1}\cup\dots\cup V_{y_n}$. Put
  $\psi=\psi_{y_1}\wedge\dots\wedge\psi_{y_n}$, clearly, $\psi\in
  CX$. Then, for all $y\in Y$,
  \begin{align*}
    \varphi_1(y)\otimes\psi(y)<u &&\text{and}&& \psi(x)=1;
  \end{align*}
  consequently,
  \begin{align*}
    \sup_{y\in X}(\psi(y)\otimes\varphi_1(y))& \le u &\text{and}&& \sup_{y\in X}(\psi(y)\otimes\varphi_2(y))&\ge\psi(x)\otimes\varphi_2(x)>u,
  \end{align*}
  and the assertion follows.
\end{proof}

\begin{corollary}\label{cor:j-is-sometimes-injective}
  Assume that the tensor product $\otimes$ on $[0,1]$ is either
  $*$, $\luk$ or $\wedge$.  For every $[0,1]^\op$-cogenerated
  separated representable $\thU$-category $X$, $j_X:VX\to GCX$ is
  injective.
\end{corollary}

\begin{corollary}\label{d:cor:1}
  Assume that the tensor product $\otimes$ on $[0,1]$ is either
  $*$, $\luk$ or $\wedge$.  Let $X$ be a $[0,1]^\op$-cogenerated
  separated representable $\thU$-category. Then $j_X:VX\to GCX$ is an
  embedding in $\Cats{[0,1]}$ and therefore also in
  $\Reps{\thU}$. Consequently, the cone
  $(\psi^\Diamond:VX\to[0,1]^\op)_{\psi\in CX}$ is point-separating
  and initial; and therefore $VX$ is $[0,1]^\op$-cogenerated.
\end{corollary}
\begin{proof}
  Since $j_X$ is injective by
  Corollary~\ref{cor:j-is-sometimes-injective}, the inequality in
  (\ref{item:7}) of Proposition~\ref{prop:j-adjunction} is actually an
  equality; and therefore $j_X$ is a split mono in $\Cats{[0,1]}$.
\end{proof}

We write $\Reps{\thU}_{[0,1]^\op}$ and $\RepDists{\thU}_{[0,1]^\op}$
to denote the full subcategory of $\Reps{\thU}$ respectively
$\RepDists{\thU}$ defined by the $[0,1]^\op$-cogenerated separated
representable $\thU$-categories. Under the conditions of
Corollary~\ref{d:cor:1}, the monad $\mV$ can be restricted to
$\Reps{\thU}_{[0,1]^\op}$ and then $\RepDists{\thU}_{[0,1]^\op}$ is
the Kleisli category for this monad on $\Reps{\thU}_{[0,1]^\op}$.

In the remainder of this section, our arguments use the continuity of
the map $\hom(u,-):[0,1]_e\to[0,1]_e$, for all
$u\in[0,1]$. Unfortunately, this property excludes $\otimes=\wedge$;
which leaves us with only two choices for $\otimes$:
\begin{itemize}
\item $u\otimes v=u*v$ is multiplication, here
  $\hom(u,v)=v\varoslash u$ is truncated division where
  $0\varoslash 0=1$; and
\item $u\otimes v=u\luk v=\max(0,u+v-1)$ is the \L{}ukasiewicz tensor,
  here $\hom(u,v)=1-\max(0,u-v)$.
\end{itemize}

\begin{lemma}\label{lem:simplify-varphi}
  Assume that $\otimes=*$ is multiplication or $\otimes=\luk$ is
  the \L{}ukasiewicz tensor. Then, for every separated representable
  $\thU$-category $X$, every $\Phi:CX\to[0,1]$ in $\FinSup{[0,1]}$ and
  every $x\in X$,
  \[
    \inf_{\psi\in CX}\hom(\psi(x),\Phi(\psi))=\inf_{\psi\in
      CX,\psi(x)=1}\Phi(\psi).
  \]
\end{lemma}
\begin{proof}
  Clearly, the left-hand side is smaller or equal to the right-hand
  side. Let now $\psi\in CX$ and put $u=\psi(x)$. Then
  $(\psi\pitchfork u)(x)=\hom(u,\psi(x))=1$ and
  \[
    \Phi(\psi\pitchfork
    u)\le\hom(u,\Phi(\psi))=\hom(\psi(x),\Phi(\psi)).\qedhere
  \]
\end{proof}

\begin{proposition}\label{prop:second-inequality}
  Assume that $\otimes=*$ is multiplication or $\otimes=\luk$ is
  the \L{}ukasiewicz tensor. Then, for every separated representable
  $\thU$-category $X$, the second inequality of
  Proposition~\ref{prop:j-adjunction} is actually an equality.
\end{proposition}
\begin{proof}
  Let $\Phi:CX\to[0,1]$ in $\FinSup{[0,1]}$ and $\psi\in CX$. Put
  $u_0=\bigvee_{x\in X}\psi(x)\otimes\inf_{\psi'\in
    CX}\hom(\psi'(x),\Phi(\psi'))$ and consider $u_0<u$. Let $x\in
  X$. Then
  \[
    u>\psi(x)\otimes\inf_{\psi'\in
      CX}\hom(\psi'(x),\Phi(\psi'))=\inf_{\psi'\in
      CX,\psi'(x)=1}\psi(x)\otimes\Phi(\psi'),
  \]
  hence there is some $\psi'\in CX$ with $\psi'(x)=1$ and
  \[
    \psi(x)\otimes\Phi(\psi')<u.
  \]
  Let now $\alpha<1$. For every $\psi'\in CX$, we put
  \[
    U_\alpha(\psi')=\{x\in X\mid
    \psi(x)\otimes\Phi(\psi')<u\}\cap\{x\in X\mid\psi'(x)>\alpha\}.
  \]
  Then $U_\alpha(\psi')$ is open in the compact Hausdorff topology of
  $X$, and
  \[
    X=\bigcup_{\psi'\in CX}U_\alpha(\psi').
  \]
  Since $X$ is compact, we find $\psi_1,\dots,\psi_n$ so that
  \[
    X=U_\alpha(\psi_1)\cup\dots\cup U_\alpha(\psi_n).
  \]
  For every $i\in\{1,\dots,n\}$ we put
  $D_i=\{x\in X\mid \psi(x)\otimes\Phi(\psi_i)\ge u\}$, then
  $U_\alpha(\psi_i)\cap D_i=\varnothing$.  Let
  $\widehat{\psi}_i:X\to[0,1]$ be a function (not necessarily a
  morphism) which is constant $1$ on $U_\alpha(\psi_i)$ and constant
  $0$ on $D_i$. Then, for all $x\in X$,
  \[
    \alpha\otimes\psi(x)
    \le(\widehat{\psi}_1(x)\otimes\psi_1(x)\otimes\psi(x))\vee\dots\vee(\widehat{\psi}_n(x)\otimes\psi_n(x)\otimes\psi(x)).
  \]
  For every $i\in\{1,\dots,n\}$ we put
  $w_i=\sup_{x\in X}\widehat{\psi}_i(x)\otimes\psi(x)$; with the
  inequality above we get
  \[
    \alpha\otimes\psi\le (w_1\otimes\psi_1)\vee\dots\vee (w_n\otimes
    \psi_n).
  \]
  Let now $i\in\{1,\dots,n\}$. Then, for every $x\in X$,
  \[
    \widehat{\psi}_i(x)\otimes\psi(x)\otimes\Phi(\psi_i)\le u,
  \]
  and therefore $w_i\otimes\Phi(\psi_i)\le u$. Consequently,
  $\alpha\otimes\Phi(\psi)\le u$ for all $\alpha<1$ and $u>u_0$; which
  implies $\Phi(\psi)\le u_0$.
\end{proof}

From Proposition~\ref{prop:second-inequality} we obtain immediately:

\begin{theorem}\label{thm:enriched-C-ff}
  For $\otimes=*$ multiplication or $\otimes=\luk$ the \L{}ukasiewicz
  tensor, the functor
  \[
    C:\RepDists{\thU}_{[0,1]^\op}\longrightarrow\FinSup{[0,1]}^\op
  \]
  is fully faithful.
\end{theorem}

Our next aim is to identify those morphisms in $\FinSup{[0,1]}$ which
correspond to ordinary relations between partially ordered compact
spaces on the other side. Recall from
Remark~\ref{rem:tensor-in-metric-case} that, for $X$ being a separated
ordered compact Hausdoff space, we can still consider
$\psi_1\otimes\psi_2$ in $CX$.

\begin{proposition}
  For $\otimes=*$ being multiplication or $\otimes=\luk$ the
  \L{}ukasiewicz tensor, a morphism $\varphi:X\modto Y$ in
  $\RepDists{\thU}$ between partially ordered compact spaces is in
  $\POSCHDIST$ if and only if $C\varphi$ satisfies \LTen.
\end{proposition}
\begin{proof}
  Clearly, if $\varphi:X\modto Y$ is in $\POSCHDIST$, then $C\varphi$
  satisfies \LTen. To see the reverse implication, note first that
  $u\in\{0,1\}$ precisely when $u\le u\otimes u$. It is enough to
  consider the case $\varphi:1\modto X$, and assume now that the
  corresponding $\Phi:CX\to[0,1]$ satisfies \LTen. Then, for all
  $x\in X$,
  \begin{align*}
    \varphi(x)\otimes\varphi(x)
    &=\left(\inf_{\psi\in CX,\psi(x)=1}\Phi(\psi)\right)\otimes\left(\inf_{\psi'\in CX,\psi'(x)=1}\Phi(\psi')\right)\\
    &=\inf_{\substack{\psi\in CX,\psi(x)=1\\ \psi'\in CX,\psi'(x)=1}}\Phi(\psi)\otimes\Phi(\psi')\\
    &\ge\inf_{\substack{\psi\in CX,\psi(x)=1\\ \psi'\in CX,\psi'(x)=1}}\Phi(\psi\otimes\psi')\\
    &=\inf_{\psi\in CX,\psi(x)=1}\Phi(\psi)=\varphi(x).\qedhere
  \end{align*}
\end{proof}

The proposition above together with Theorem~\ref{thm:enriched-C-ff} is
certainly related to Theorem~\ref{thm:full-embedding-1}; however, in
Section~\ref{sec:duality-distributors} we consider finitely
cocomplete $\V$-categories equipped with an \emph{additional} monoid 
structure which is not needed in this section. In fact,
Theorem~\ref{thm:enriched-C-ff} allows us to characterise the
multiplication $\otimes$ of $CX$ within $\FinSup{[0,1]}$.

\begin{lemma}
  Assume that $\otimes=*$ is multiplication or $\otimes=\luk$ is
  the \L{}ukasiewicz tensor. Let $X$ be a partially ordered compact
  space and let $\psi_0\in CX$. Let $\Phi:CX\to CX$ in
  $\FinSup{[0,1]}$ with $\Phi(1)\le\psi_0$ and $\Phi(\psi)\le\psi$,
  for all $\psi\in CX$. Then $\Phi=\psi_0\otimes-$ provided that
  $\psi_0\otimes\psi\le\Phi(\psi)$, for all $\psi\in CX$.
\end{lemma}
\begin{proof}
  Let $x\in X$ and consider
  \[
    CX\xrightarrow{\;\Phi\,}CX\xrightarrow{\,\ev_x\,}[0,1]
  \]
  in $\FinSup{[0,1]}$. By Theorem~\ref{thm:enriched-C-ff}, this arrow
  corresponds to the continuous $[0,1]$-distributor
  $\varphi:1\modto X$ given by
  \[
    \varphi(y)=\inf_{\psi\in CX}\hom(\psi(y),\Phi(\psi)(x)),
  \]
  for all $y\in X$. Let now $y\in X$, we consider the following two
  cases.
  \begin{description}
  \item[$y\not\ge x$] By Proposition~\ref{prop:initial-cogenerator},
    there exists some $\psi\in CX$ with $\psi(y)=1$ and
    $\psi(x)=0$. Since $\Phi(\psi)\le\psi$, we obtain
    $\Phi(\psi)(x)\le\psi(x)=0$; hence $\varphi(y)=0$.
  \item[$y\ge x$] Firstly, for every $\psi\in CX$,
    \[
      \hom(\psi(y),\Phi(\psi)(x))\ge\hom(\psi(x),\Phi(\psi)(x))\ge\psi_0(x)
    \]
    since $\psi_0(x)\otimes\psi(x)\le\Phi(\psi)(x)$. Secondly,
    $\hom(1,\Phi(1)(x))\le\psi_0(x)$. Therefore
    $\varphi(y)=\psi_0(x)$.
  \end{description}
  Finally, we obtain
  \[
    \Phi(\psi)(x)=\sup_{y\in X}(\varphi(y)\otimes\psi(y))=\sup_{y\ge
      x}(\psi_0(x)\otimes\psi(y))=\psi_0(x)\otimes\psi(x),
  \]
  for all $\psi\in CX$.
\end{proof}

\begin{theorem}
  Assume that $\otimes=*$ is multiplication or $\otimes=\luk$ is
  the \L{}ukasiewicz tensor. Let $X$ be a partially ordered compact
  space. Then, for every $\psi_0\in CX$, $\psi_0\otimes-:CX\to CX$ is
  the largest morphism $\Phi:CX\to CX$ in $\FinSup{[0,1]}$ satisfying
  \begin{align}\label{eq:tens-properties}
    \Phi(1)\le\psi_0 &&\text{and}&& \Phi(\psi)\le\psi,\text{ for all }\psi\in CX.
  \end{align}
\end{theorem}
\begin{proof}
  Clearly, $\psi_0\otimes-$ satisfies \eqref{eq:tens-properties}, and
  the lemma above tells us already that it is maximal among all those
  maps. Let now $\Phi_1,\Phi_2:CX\to CX$ be in $\FinSup{[0,1]}$
  satisfying \eqref{eq:tens-properties}. Then also the composite arrow
  \[
    CX\xrightarrow{\,\langle \Phi_1,\Phi_2\rangle\,}CX\times
    CX\xrightarrow{\,\vee\,}CX
  \]
  satisfies \eqref{eq:tens-properties}, therefore the collection of
  all morphism $\Phi:CX\to CX$ in $\FinSup{[0,1]}$ satisfying
  \eqref{eq:tens-properties} is directed. Consequently,
  $\psi_0\otimes-:CX\to CX$ is the largest such morphism.
\end{proof}


\begin{thebibliography}{63}\setlength{\itemsep}{1ex}\small
\providecommand{\natexlab}[1]{#1}
\providecommand{\url}[1]{\texttt{#1}}
\providecommand{\urlprefix}{URL }
\providecommand{\eprint}[2][]{\url{#2}}

\bibitem[{Ad{\'a}mek \emph{et~al.}(1990)Ad{\'a}mek, Herrlich and
  Strecker}]{AHS90}
\textsc{Ad{\'a}mek, J.}, \textsc{Herrlich, H.} and \textsc{Strecker, G.~E.}
  (1990), \emph{Abstract and concrete categories: {T}he joy of cats}, Pure and
  Applied Mathematics (New York), John Wiley \& Sons Inc., New York, xiv + 482
  pages, {Republished in: Reprints in Theory and Applications of Categories,
  No. 17 (2006) pp.~1--507}.

\bibitem[{Ad{\'a}mek and Rosick{\'y}(1994)}]{AR94}
\textsc{Ad{\'a}mek, J.} and \textsc{Rosick{\'y}, J.} (1994), \emph{Locally
  presentable and accessible categories}, volume 189 of \emph{London
  Mathematical Society Lecture Note Series}, Cambridge University Press,
  Cambridge, xiv + 316 pages.

\bibitem[{Ad{\'a}mek \emph{et~al.}(2010)Ad{\'a}mek, Rosick{\'y} and
  Vitale}]{ARV10}
\textsc{Ad{\'a}mek, J.}, \textsc{Rosick{\'y}, J.} and \textsc{Vitale, E.}
  (2010), \emph{Algebraic Theories}, Cambridge University Press.

\bibitem[{Alsina \emph{et~al.}(2006)Alsina, Frank and Schweizer}]{AFS06}
\textsc{Alsina, C.}, \textsc{Frank, M.~J.} and \textsc{Schweizer, B.} (2006),
  \emph{Associative functions. Triangular norms and copulas}, Hackensack, NJ:
  World Scientific, xiv + 237 pages.

\bibitem[{Baez and Dolan(2001)}]{BD01}
\textsc{Baez, J.} and \textsc{Dolan, J.} (2001), From finite sets to {F}eynman
  diagrams, in \emph{{Mathematics Unlimited - 2001 and Beyond}}, volume~1,
  pages 29--50, Springer, Berlin,
  \href{http://arxiv.org/abs/0004133}{{\ttfamily arXiv:0004133 [math.QA]}}.

\bibitem[{Banaschewski(1983)}]{Ban83}
\textsc{Banaschewski, B.} (1983), On lattices of continuous functions,
  \emph{Quaestiones Mathematic{\ae}} \textbf{6}~(1-3), 1--12.

\bibitem[{Banaschewski \emph{et~al.}(2006)Banaschewski, Lowen and
  Van~Olmen}]{BLO06}
\textsc{Banaschewski, B.}, \textsc{Lowen, R.} and \textsc{Van~Olmen, C.}
  (2006), Sober approach spaces, \emph{Topology and its Applications}
  \textbf{153}~(16), 3059--3070.

\bibitem[{Bonsangue \emph{et~al.}(2007)Bonsangue, Kurz and Rewitzky}]{BKR07}
\textsc{Bonsangue, M.~M.}, \textsc{Kurz, A.} and \textsc{Rewitzky, I.~M.}
  (2007), Coalgebraic representations of distributive lattices with operators,
  \emph{Topology and its Applications} \textbf{154}~(4), 778--791.

\bibitem[{Bourbaki(1966)}]{Bou66}
\textsc{Bourbaki, N.} (1966), \emph{General topology, part {I}}, Hermann, Paris
  and Addison-Wesley, chapters 1--4.

\bibitem[{Clark and Davey(1998)}]{CD98}
\textsc{Clark, D.~M.} and \textsc{Davey, B.~A.} (1998), \emph{Natural dualities
  for the working algebraist}, volume~57 of \emph{Cambridge Studies in Advanced
  Mathematics}, Cambridge University Press, Cambridge, xii + 356 pages.

\bibitem[{Clementino and Hofmann(2009)}]{CH09a}
\textsc{Clementino, M.~M.} and \textsc{Hofmann, D.} (2009), Relative
  injectivity as cocompleteness for a class of distributors, \emph{Theory and
  Applications of Categories} \textbf{21}~(12), 210--230,
  \href{http://arxiv.org/abs/0807.4123}{{\ttfamily arXiv:0807.4123 [math.CT]}}.

\bibitem[{Dimov and Tholen(1989)}]{DT89}
\textsc{Dimov, G.~D.} and \textsc{Tholen, W.} (1989), A characterization of
  representable dualities, in \emph{Categorical topology and its relation to
  analysis, algebra and combinatorics ({P}rague, 1988)}, pages 336--357, World
  Sci. Publ., Teaneck, NJ.

\bibitem[{Eilenberg and Kelly(1966)}]{EK66}
\textsc{Eilenberg, S.} and \textsc{Kelly, G.~M.} (1966), Closed categories, in
  \emph{Proc. Conf. Categorical Algebra (La Jolla, Calif., 1965)}, pages
  421--562, Springer, New York.

\bibitem[{Faucett(1955)}]{Fau55}
\textsc{Faucett, W.~M.} (1955), Compact semigroups irreducibly connected
  between two idempotents, \emph{Proceedings of the American Mathematical
  Society} \textbf{6}~(5), 741--747.

\bibitem[{Flagg(1992)}]{Fla92}
\textsc{Flagg, R.~C.} (1992), Completeness in continuity spaces, {Seely, R. A.
  G. (ed.), Category theory 1991. Proceedings of an international summer
  category theory meeting, held in Montr{\'e}al, Qu{\'e}bec, Canada, June
  23-30, 1991. Providence, RI: American Mathematical Society. CMS Conf. Proc.
  13, 183-199 (1992).}

\bibitem[{Gelfand(1941)}]{Gel41a}
\textsc{Gelfand, I.} (1941), {Normierte Ringe}, \emph{Recueil
  Math{\'{e}}matique. Nouvelle S{\'{e}}rie} \textbf{9}~(1), 3--24.

\bibitem[{Gierz \emph{et~al.}(1980)Gierz, Hofmann, Keimel, Lawson, Mislove and
  Scott}]{GHK+80}
\textsc{Gierz, G.}, \textsc{Hofmann, K.~H.}, \textsc{Keimel, K.},
  \textsc{Lawson, J.~D.}, \textsc{Mislove, M.~W.} and \textsc{Scott, D.~S.}
  (1980), \emph{A compendium of continuous lattices}, Springer-Verlag, Berlin,
  xx + 371 pages.

\bibitem[{Gierz \emph{et~al.}(2003)Gierz, Hofmann, Keimel, Lawson, Mislove and
  Scott}]{GHK+03}
\textsc{Gierz, G.}, \textsc{Hofmann, K.~H.}, \textsc{Keimel, K.},
  \textsc{Lawson, J.~D.}, \textsc{Mislove, M.~W.} and \textsc{Scott, D.~S.}
  (2003), \emph{Continuous lattices and domains}, volume~93 of
  \emph{Encyclopedia of Mathematics and its Applications}, Cambridge University
  Press, Cambridge, xxxvi+591 pages.

\bibitem[{Halmos(1956)}]{Hal56}
\textsc{Halmos, P.~R.} (1956), Algebraic logic. {I}. {M}onadic {B}oolean
  algebras, \emph{Compositio Mathematica} \textbf{12}, 217--249.

\bibitem[{Hofmann(2007)}]{Hof07}
\textsc{Hofmann, D.} (2007), Topological theories and closed objects,
  \emph{Advances in Mathematics} \textbf{215}~(2), 789--824.

\bibitem[{Hofmann(2013)}]{Hof13}
\textsc{Hofmann, D.} (2013), Duality for distributive spaces, \emph{Theory and
  Applications of Categories} \textbf{28}~(3), 66--122,
  \href{http://arxiv.org/abs/1009.3892}{{\ttfamily arXiv:1009.3892 [math.CT]}}.

\bibitem[{Hofmann(2014)}]{Hof14}
\textsc{Hofmann, D.} (2014), The enriched {V}ietoris monad on representable
  spaces, \emph{Journal of Pure and Applied Algebra} \textbf{218}~(12),
  2274--2318, \href{http://arxiv.org/abs/1212.5539}{{\ttfamily arXiv:1212.5539
  [math.CT]}}.

\bibitem[{Hofmann and Nora(2015)}]{HN15}
\textsc{Hofmann, D.} and \textsc{Nora, P.} (2015), Dualities for modal algebras
  from the point of view of triples, \emph{Algebra Universalis}
  \textbf{73}~(3), 297--320, \href{http://arxiv.org/abs/1302.5609}{{\ttfamily
  arXiv:1302.5609 [math.LO]}}.

\bibitem[{Hofmann \emph{et~al.}(2014)Hofmann, Seal and Tholen}]{HST14}
\textsc{Hofmann, D.}, \textsc{Seal, G.~J.} and \textsc{Tholen, W.}, editors
  (2014), \emph{Monoidal Topology. A Categorical Approach to Order, Metric, and
  Topology}, Cambridge University Press, 518 pages, authors: Maria Manuel
  Clementino, Eva Colebunders, Dirk Hofmann, Robert Lowen, Rory
  Lucyshyn-Wright, Gavin J. Seal and Walter Tholen.

\bibitem[{Hofmann and Stubbe(2011)}]{HS11}
\textsc{Hofmann, D.} and \textsc{Stubbe, I.} (2011), Towards {S}tone duality
  for topological theories, \emph{Topology and its Applications}
  \textbf{158}~(7), 913--925, \href{http://arxiv.org/abs/1004.0160}{{\ttfamily
  arXiv:1004.0160 [math.CT]}}.

\bibitem[{Hofmann and Tholen(2010)}]{HT10}
\textsc{Hofmann, D.} and \textsc{Tholen, W.} (2010), {L}awvere completion and
  separation via closure, \emph{Applied Categorical Structures}
  \textbf{18}~(3), 259--287, \href{http://arxiv.org/abs/0801.0199}{{\ttfamily
  arXiv:0801.0199 [math.CT]}}.

\bibitem[{Isbell(1982)}]{Isb82}
\textsc{Isbell, J.~R.} (1982), Generating the algebraic theory of {$C(X)$},
  \emph{Algebra Universalis} \textbf{15}~(2), 153--155.

\bibitem[{Johnstone(1986)}]{Joh86}
\textsc{Johnstone, P.~T.} (1986), \emph{Stone spaces}, volume~3 of
  \emph{Cambridge Studies in Advanced Mathematics}, Cambridge University Press,
  Cambridge, xxii + 370 pages, reprint of the 1982 edition.

\bibitem[{Kaplansky(1947)}]{Kap47}
\textsc{Kaplansky, I.} (1947), Lattices of continuous functions, \emph{Bulletin
  of the American Mathematical Society} \textbf{53}~(6), 617--623.

\bibitem[{Kaplansky(1948)}]{Kap48}
\textsc{Kaplansky, I.} (1948), Lattices of continuous functions {II},
  \emph{American Journal of Mathematics} \textbf{70}~(3), 626--634.

\bibitem[{Kegelmann(1999)}]{Keg99}
\textsc{Kegelmann, M.} (1999), \emph{Continuous domains in logical form}, Ph.D.
  thesis, School of Computer Science, The University of Birmingham.

\bibitem[{Kelly(1982)}]{Kel82}
\textsc{Kelly, G.~M.} (1982), \emph{Basic concepts of enriched category
  theory}, volume~64 of \emph{London Mathematical Society Lecture Note Series},
  Cambridge University Press, Cambridge, 245 pages, {Republished in: Reprints
  in Theory and Applications of Categories. No.~10 (2005), 1--136}.

\bibitem[{Kelly and Lack(2000)}]{KL00}
\textsc{Kelly, G.~M.} and \textsc{Lack, S.} (2000), On the monadicity of
  categories with chosen colimits, \emph{Theory and Applications of Categories}
  \textbf{7}~(7), 148--170.

\bibitem[{Kelly and Schmitt(2005)}]{KS05}
\textsc{Kelly, G.~M.} and \textsc{Schmitt, V.} (2005), Notes on enriched
  categories with colimits of some class, \emph{Theory and Applications of
  Categories} \textbf{14}~(17), 399--423.

\bibitem[{Kock(1971)}]{Koc71b}
\textsc{Kock, A.} (1971), On double dualization monads, \emph{Mathematica
  Scandinavica} \textbf{27}~(2), 151--165.

\bibitem[{Kupke \emph{et~al.}(2004)Kupke, Kurz and Venema}]{KKV04}
\textsc{Kupke, C.}, \textsc{Kurz, A.} and \textsc{Venema, Y.} (2004), Stone
  coalgebras, \emph{Theoretical Computer Science} \textbf{327}~(1-2), 109--134.

\bibitem[{Lambek and Rattray(1978)}]{LR78}
\textsc{Lambek, J.} and \textsc{Rattray, B.~A.} (1978), Functional completeness
  and {S}tone duality, in \emph{Studies in Foundations and Combinatorics},
  volume~1, pages 1--9, Academic Press, New York.

\bibitem[{Lambek and Rattray(1979)}]{LR79}
\textsc{Lambek, J.} and \textsc{Rattray, B.~A.} (1979), A general
  {S}tone--{G}elfand duality, \emph{Transactions of the American Mathematical
  Society} \textbf{248}~(1), 1--35.

\bibitem[{Lawvere(1973)}]{Law73}
\textsc{Lawvere, F.~W.} (1973), Metric spaces, generalized logic, and closed
  categories, \emph{Rendiconti del Seminario Matem{\`{a}}tico e Fisico di
  Milano} \textbf{43}~(1), 135--166, {Republished in: Reprints in Theory and
  Applications of Categories, No. 1 (2002), 1--37}.

\bibitem[{Lowen(1997)}]{Low97}
\textsc{Lowen, R.} (1997), \emph{Approach spaces}, Oxford Mathematical
  Monographs, The Clarendon Press Oxford University Press, New York, x + 253
  pages, the missing link in the topology-uniformity-metric triad, Oxford
  Science Publications.

\bibitem[{Marra and Reggio(2017)}]{MR17}
\textsc{Marra, V.} and \textsc{Reggio, L.} (2017), Stone duality above
  dimension zero: {Axiomatising} the algebraic theory of {C}({X}),
  \emph{Advances in Mathematics} \textbf{307}, 253--287,
  \href{http://arxiv.org/abs/1508.07750}{{\ttfamily arXiv:1508.07750
  [math.LO]}}.

\bibitem[{Mostert and Shields(1957)}]{MS57}
\textsc{Mostert, P.~S.} and \textsc{Shields, A.~L.} (1957), On the structure of
  semi-groups on a compact manifold with boundary, \emph{Annals of Mathematics.
  Second Series} \textbf{65}~(1), 117--143.

\bibitem[{Nachbin(1950)}]{Nac50}
\textsc{Nachbin, L.} (1950), \emph{Topologia e Ordem}, Univ. of Chicago Press.

\bibitem[{Nachbin(1965)}]{Nac65}
\textsc{Nachbin, L.} (1965), \emph{Topology and order}, Translated from the
  Portuguese by Lulu Bechtolsheim. Van Nostrand Mathematical Studies, No. 4, D.
  Van Nostrand Co., Inc., Princeton, N.J.-Toronto, Ont.-London, vi + 122 pages.

\bibitem[{Porst and Tholen(1991)}]{PT91}
\textsc{Porst, H.-E.} and \textsc{Tholen, W.} (1991), Concrete dualities, in
  H.~Herrlich and H.-E. Porst, editors, \emph{Category theory at work},
  volume~18 of \emph{Res. Exp. Math.}, pages 111--136, Heldermann Verlag,
  Berlin.

\bibitem[{Pumpl{\"u}n(1970)}]{Pum70}
\textsc{Pumpl{\"u}n, D.} (1970), Eine {B}emerkung \"uber {M}onaden und
  adjungierte {F}unktoren, \emph{Mathematische Annalen} \textbf{185}~(4),
  329--337.

\bibitem[{Radul(1997)}]{Rad97}
\textsc{Radul, T.} (1997), A functional representation of the hyperspace monad,
  \emph{Commentat. Math. Univ. Carol} \textbf{38}~(1), 165--168.

\bibitem[{Radul(2009)}]{Rad09}
\textsc{Radul, T.} (2009), Hyperspace as intersection of inclusion hyperspaces
  and idempotent measures, \emph{Matematychni Studi{\"{i}}} \textbf{31}~(2),
  207--210.

\bibitem[{Raney(1952)}]{Ran52}
\textsc{Raney, G.~N.} (1952), Completely distributive complete lattices,
  \emph{Proceedings of the American Mathematical Society} \textbf{3}~(5),
  677--680.

\bibitem[{Schalk(1993)}]{Sch93}
\textsc{Schalk, A.} (1993), \emph{Algebras for Generalized Power
  Constructions}, Ph.D. thesis, Technische Hochschule Darmstadt.

\bibitem[{Shapiro(1992)}]{Sha92}
\textsc{Shapiro, L.~B.} (1992), On function extension operators and normal
  functors, \emph{Vestnik Moskovskogo Universiteta. Seriya I. Matematika,
  Mekhanika} ~(1), 35--42.

\bibitem[{Stone(1936)}]{Sto36}
\textsc{Stone, M.~H.} (1936), The theory of representations for {B}oolean
  algebras, \emph{Transactions of the American Mathematical Society}
  \textbf{40}~(1), 37--111.

\bibitem[{Stone(1938{\natexlab{a}})}]{Sto38a}
\textsc{Stone, M.~H.} (1938{\natexlab{a}}), The representation of boolean
  algebras, \emph{Bulletin of the American Mathematical Society}
  \textbf{44}~(12), 807--816.

\bibitem[{Stone(1938{\natexlab{b}})}]{Sto38}
\textsc{Stone, M.~H.} (1938{\natexlab{b}}), Topological representations of
  distributive lattices and {B}rouwerian logics, \emph{\v{C}asopis pro
  p\v{e}stov\'an\'i matematiky a fysiky} \textbf{67}~(1), 1--25, eprint:
  \url{http://dml.cz/handle/10338.dmlcz/124080}.

\bibitem[{Stone(1948)}]{Sto48}
\textsc{Stone, M.~H.} (1948), The generalized weierstrass approximation
  theorem, \emph{Mathematics Magazine} \textbf{21}~(5), 237--254.

\bibitem[{Stubbe(2005)}]{Stu05}
\textsc{Stubbe, I.} (2005), Categorical structures enriched in a quantaloid:
  categories, distributors and functors, \emph{Theory and Applications of
  Categories} \textbf{14}~(1), 1--45,
  \href{http://arxiv.org/abs/0409473}{{\ttfamily arXiv:0409473 [math.CT]}}.

\bibitem[{Stubbe(2006)}]{Stu06}
\textsc{Stubbe, I.} (2006), Categorical structures enriched in a quantaloid:
  tensored and cotensored categories, \emph{Theory and Applications of
  Categories} \textbf{16}~(14), 283--306,
  \href{http://arxiv.org/abs/0411366}{{\ttfamily arXiv:0411366 [math.CT]}}.

\bibitem[{Stubbe(2007)}]{Stu07a}
\textsc{Stubbe, I.} (2007), $\mathcal{Q}$-modules are
  $\mathcal{Q}$--suplattices, \emph{Theory and Applications of Categories}
  \textbf{19}~(4), 50--60, \href{http://arxiv.org/abs/0809.4343}{{\ttfamily
  arXiv:0809.4343 [math.CT]}}.

\bibitem[{Tholen(2009)}]{Tho09}
\textsc{Tholen, W.} (2009), Ordered topological structures, \emph{Topology and
  its Applications} \textbf{156}~(12), 2148--2157.

\bibitem[{Van~Olmen(2005)}]{Olm05}
\textsc{Van~Olmen, C.} (2005), \emph{A study of the interaction between frame
  theory and approach theory}, Ph.D. thesis, University of Antwerp.

\bibitem[{Van~Olmen and Verwulgen(2010)}]{OV10}
\textsc{Van~Olmen, C.} and \textsc{Verwulgen, S.} (2010), A finite axiom scheme
  for approach frames, \emph{Bulletin of the Belgian Mathematical Society -
  Simon Stevin} \textbf{17}~(5), 899--908.

\bibitem[{Vietoris(1922)}]{Vie22}
\textsc{Vietoris, L.} (1922), Bereiche zweiter {O}rdnung, \emph{Monatshefte
  f{\"u}r Mathematik und Physik} \textbf{32}~(1), 258--280.

\bibitem[{Wood(2004)}]{Woo04}
\textsc{Wood, R.~J.} (2004), {Ordered Sets via Adjunction}, in M.~C. Pedicchio
  and W.~Tholen, editors, \emph{Categorical Foundations: Special Topics in
  Order, Topology, Algebra, and Sheaf Theory}, volume~97 of \emph{Encyclopedia
  Math. Appl.}, pages 5--47, Cambridge University Press ({CUP}), Cambridge.

\end{thebibliography}

\end{document}